\documentclass[12pt,a4paper,reqno]{amsart}
        
\usepackage{amsmath}
\usepackage{amssymb,latexsym}
\usepackage{verbatim}
\usepackage{extarrows}
\usepackage{enumerate}
\usepackage{txfonts}
\usepackage{mathtools}
\usepackage{bm}
\usepackage{mathtools}
\usepackage[tableposition=top]{caption}
\usepackage{booktabs,dcolumn}
\usepackage{amsthm}
\usepackage{setspace}
\usepackage{xcolor}
\usepackage{tikz-cd}
\usepackage{quiver}
\usepackage{amsfonts}
\usepackage[letterpaper, portrait, margin=1.2in]{geometry}

\makeatletter
\@namedef{subjclassname@2020}{  \textup{2020} Mathematics Subject Classification}
\makeatother

\theoremstyle{plain}
\newtheorem{theorem}{Theorem}[section]

\newtheorem{lemma}[theorem]{Lemma}
\newtheorem{corollary}[theorem]{Corollary}

\theoremstyle{definition}
\newtheorem{definition}[theorem]{Definition}
\newtheorem{remark}[theorem]{Remark}

\def\acts{\curvearrowright}
\newcommand\op{{\operatorname{op}}}

\newcommand\CHProb{{\mathbf{CHPrb}}}

\newcommand\OpProbAlg{\mathbf{PrbAlg}}

\newcommand\Hilb{\mathbf{Hilb}}
\newcommand\Mes{\mathtt{Alg}}
\newcommand\Stone{\mathtt{Conc}}
\newcommand\Abs{\mathtt{Abs}}

\email{ajamneshan@ku.edu.tr}
\email{minghaopan@ucla.edu}

\begin{document}
\title[]{Uniform syndeticity in multiple recurrence}
\author[A. Jamneshan]{Asgar Jamneshan}
\address{Department of Mathematics\\
Ko\c{c} University \\
\.{I}stanbul, T\"urkiye}
\author[M. Pan]{Minghao Pan}
\address{The Division of Physics, Mathematics and Astronomy, Caltech, Pasadena, California, USA}
\date{\today }



\begin{abstract}
The main theorem of this paper establishes a uniform syndeticity result concerning the multiple recurrence of measure-preserving actions on probability spaces. More precisely, for any integers $d,l\geq 1$ and any $\varepsilon > 0$, we prove the existence of $\delta>0$ and $K\geq 1$ (dependent only on $d$, $l$, and $\varepsilon$) such that the following holds:

Consider a solvable group $\Gamma$ of derived length $l$, a probability space $(X, \mu)$, and $d$ pairwise commuting measure-preserving $\Gamma$-actions $T_1, \ldots, T_d$ on $(X, \mu)$. Let $E$ be a measurable set in $X$ with $\mu(E) \geq \varepsilon$. Then, $K$ many (left) translates of 
\begin{equation*}
\left\{\gamma\in\Gamma\colon
\mu(T_1^{\gamma^{-1}}(E)\cap T_2^{\gamma^{-1}} \circ
T^{\gamma^{-1}}_1(E)\cap \cdots \cap T^{\gamma^{-1}}_d\circ
T^{\gamma^{-1}}_{d-1}\circ \ldots \circ T^{\gamma^{-1}}_1(E))\geq \delta
\right\}
\end{equation*}
cover $\Gamma$. This result extends and refines uniformity results by Furstenberg and Katznelson.  

As a combinatorial application, we obtain the following uniformity result. For any integers $d,l\geq 1$ and any $\varepsilon > 0$, there are $\delta>0$ and $K\geq 1$ (dependent only on $d$, $l$, and $\varepsilon$) such that for all finite solvable groups $G$ of derived length $l$ and any subset $E\subset G^d$ with $m^{\otimes d}(E)\geq \varepsilon$ (where $m$ is the uniform measure on $G$), we have that $K$-many (left) translates of
\begin{multline*}
    \{g\in G\colon m^{\otimes d}(\{(a_1,\ldots,a_n)\in G^d\colon \\ (a_1,\ldots,a_n),(ga_1,a_2,\ldots,a_n),\ldots,(ga_1,ga_2,\ldots, ga_n)\in E\})\geq \delta \}
\end{multline*}
cover $G$.

The proof of our main result is a consequence of an ultralimit version of Austin's amenable ergodic Szem\'eredi theorem. 
\end{abstract}

\maketitle

\section{Introduction}

A subset of a group is called syndetic\footnote{Throughout, F\o lner nets are left F\o lner nets, and syndetic sets are left syndetic sets in non-commutative groups.} if the union of finitely many translates of it cover the whole group. 

One version \cite{furstenberg1985ergodic} of Furstenburg's multiple recurrence theorem \cite{furstenberg1977ergodic} is as follows:
\begin{theorem}\label{furstenburg}
For every abelian group $\Gamma$, each probability space $(X,\mu)$ with finitely many pairwise commuting
measure-preserving $\Gamma$-actions $T_i\acts (X,\mu)$, $i=1,\ldots,d$, and
all measurable sets $E$ in $X$ with positive measure,  the return set
\begin{equation*}
\left\{\gamma\in\Gamma\colon
\mu(E\cap T^{-\gamma}_1(E)\cap T^{-\gamma}_2(E)\cap \cdots \cap T^{-\gamma}_d(E))>
0\right\}
\end{equation*}
is syndetic. 
\end{theorem}
In this paper, our objective is to explore the uniformity of multiple recurrence theorems. There exist two directions of uniformity. Firstly, can we establish a uniform lower bound for the measure of the multiple recurrence event, denoted as $\mu(E \cap T_1^\gamma(E) \cap T_2^\gamma(E) \cap \cdots \cap T_d^\gamma(E))$, keeping it away from zero? Secondly, can we assert that the return set is uniformly not too small? Increasing the value of $d$ or shrinking the measure of the set $E$ might lead to a reduction in both the multiple recurrence event and the return set. Nevertheless, our aspiration is for these measures to remain independent of certain factors: the group $\Gamma$, the probability space $(X, \mu)$, the commuting measure-preserving $\Gamma$-actions $T_1,\ldots,T_d$, and the choice of measurable set $E$ - as long as we fixed $d$ and $\mu(E)$.

To pursue the second aspect of uniformity, it becomes necessary to establish a method for quantifying the size of a subset within a group. In light of the statement of Theorem \ref{furstenburg}, a natural choice is to utilize the concept of $K$-syndeticity\footnote{In a previous version of this paper, we quantified syndeticity using the size of the lower Banach density of a subset. We are indebted to the anonymous referee for suggesting the use of the more natural (and seemingly stronger) concept of $K$-syndicity, which also had the benifit of significantly simplifying the proof of our main uniform syndeticity result.}: Given a group $\Gamma$ and an integer $K\geq 1$, a subset $S\subset \Gamma$ is said to be \emph{$K$-syndetic} if $K$ many translates of $S$ cover $\Gamma$. 

Numerous findings pertaining to uniform syndeticity are available within the existing literature. Among these, a notable contribution was made by Furstenberg and Katznelson, who demonstrated the prevalence of uniform syndeticity across all $\mathbb{Z}$-actions.

\begin{theorem}[Uniform syndeticity, $\mathbb{Z}$-case]
\label{thm-main-Z} For every integer $d\geq 1$ and any $\varepsilon>0$ there
are $\delta>0$ and $K\geq 1$ (only depending on $\varepsilon,d$) such that for
any probability space $(X,\mu)$, every $d$ many pairwise commuting
measure-preserving transformations $T_i\colon X\to X$, $i=1,\ldots,d$, and
all measurable sets $E$ in $X$ with $\mu(E)\geq \varepsilon$ it holds that 
\begin{equation*}
\left\{n\in\mathbb{Z}\colon
\mu(E\cap T^{-n}_1(E)\cap T^{-n}_2(E)\cap \cdots \cap T^{-n}_d(E))\geq
\delta \right\}
\end{equation*}
is $K$-syndetic. 
\end{theorem}

\begin{proof}
This result can be deduced from \cite[Theorem 2.1(iii)]{bergelson-uniformity}.
\end{proof}

The following weaker assertion is established for a fixed arbitrary countable abelian group by Furstenberg and Katznelson in their work \cite{furstenberg1985ergodic}. In this version of uniform syndeticity, the probability of the multiple recurrence event is \textit{not} shown to be uniformly bounded away from zero, as observed in Theorem \ref{thm-main-Z}, or later shown in Theorem \ref{thm-strength}. 

\begin{theorem}[Weak uniform syndeticity, countable abelian case]
\label{thm-main-abelian} Let $\Gamma$ be a countable abelian group. For
every integer $d\geq 1$ and any $\varepsilon>0$ there exists $K\geq 1$ (only
depending on $\varepsilon,d$ and $\Gamma$) such that for any probability space 
$(X,\mu)$, every $d$ many pairwise commuting measure-preserving actions  $T_i\colon \Gamma\acts(X,\mu)$, $i=1,\ldots,d$, and every
measurable set $E$ in $X$ with $\mu(E)\geq \varepsilon$ it holds that 
\begin{equation*}
\left\{\gamma\in\Gamma\colon \mu(E\cap
T^{-\gamma}_1(E)\cap T^{-\gamma}_2(E)\cap \cdots \cap T^{-\gamma}_d(E))>0
\right\}
\end{equation*}
is $K$-syndetic. 
\end{theorem}
\begin{proof}
The claim follows from combining the results in \cite[\S 10]{furstenberg1985ergodic}, see the last remark therein.
\end{proof}
\begin{remark}
In fact, in \cite{furstenberg1985ergodic}, Furstenberg and Katznelson establish that the return set 
$$\left\{\gamma\in\Gamma\colon \mu(E\cap
T^{-\gamma}_1(E)\cap T^{-\gamma}_2(E)\cap \cdots \cap T^{-\gamma}_d(E))>0
\right\}$$
satisfies stronger notions of largeness than syndeticity such as IP$^*$ or even IP$^*_r$. 
However, we will focus on strengthening and generalizing the slightly weaker consequence stated in Theorem \ref{thm-main-abelian}.
\end{remark}

Our main result establishes a new proof and a joint generalization and strengthening of the Theorems \ref{thm-main-Z} and \ref{thm-main-abelian} by relaxing the dependence of $\delta$ and $K$ on the acting group:
\begin{theorem}\label{thm-strength}
For all integers $d,l\geq 1$ and any $\varepsilon>0$ there exist $\delta>0$ and $K\geq 1$ (only
depending on $\varepsilon,d,l$) such that for any solvable group $\Gamma$ of derived length $l$, any probability space 
$(X,\mu)$, every $d$ many pairwise commuting measure-preserving actions $T_i\colon \Gamma\acts(X,\mu)$, $i=1,\ldots,d$, and every
measurable set $E$ in $X$ with $\mu(E)\geq \varepsilon$ it holds that 
\begin{equation*}
\left\{ \gamma \in \Gamma \colon \mu( E\cap
T_{1}^{\gamma ^{-1}}(E)\cap \left( T_{[1,2]}^{\gamma }\right) ^{-1}(E)\cap
\cdots \cap \left( T_{[1,d]}^{\gamma }\right)^{-1}(E))\geq \delta \right\}
\end{equation*}
is $K$-syndetic, where $T_{[a,b]}^{\gamma }:=T_{a}^{\gamma }\circ
T_{a+1}^{\gamma }\circ \cdots \circ T_{b}^{\gamma}$.
\end{theorem}
We recall that the derived length $n$ of a solvable group $\Gamma$ is the least $n$ for which $\Gamma^{(n)}=1$, where $\Gamma^{(i)}$ is recursively defined by $\Gamma^{(0)}=\Gamma$ and $\Gamma^{(i+1)}$ is the commutator subgroup $[\Gamma^{(i)},\Gamma^{(i)}]$ of $\Gamma^{(i)}$.

\begin{remark}
In fact, we establish a more general version of Theorem \ref{thm-strength} where we can consider any uniformly amenable class of groups, of which a class of solvable groups of fixed derived length is an example. See Section \ref{sec-prelim} for the definition of uniform amenability and Theorem \ref{thm-strength1} for the general statement.
\end{remark}

\begin{remark}
  It is important to discern the variance in the articulation of the multiple recurrence event in the abelian case in the Theorems \ref{furstenburg}, \ref{thm-main-Z}, and \ref{thm-main-abelian}, where we consider $T_i^\gamma$ rather than the composite actions $T_{[1,i]}^{\gamma}$ as presented in the formulation of Theorem \ref{thm-strength}. The possibility of attaining an analogous formulation to Theorems \ref{furstenburg}, \ref{thm-main-Z}, and \ref{thm-main-abelian} does indeed arise for nilpotent groups. However, it is crucial to note, as observed by Bergelson and Leibman in \cite{bergelson2004failure}, that such a formulation fails to hold universally for solvable groups.
\end{remark}

\begin{remark}
Quantitatively stronger results in the form of Khintchine-type bounds are available in more specific situations, as seen in \cite{abo, bhk, btz-recurrence, chu-two, lowerbound, shalom-khintchine}. To the best of our knowledge, Theorem \ref{thm-strength} is the first result of its kind to establish the existence of uniform bounds for arbitrary $d$, independent of $\Gamma$ (within a large class of groups), and without requiring the hypothesis of ergodicity. 
\end{remark}

\begin{remark}
In \cite{j-dcds}, the first author and coauthors established a slightly weaker formulation of Theorem \ref{thm-strength} (in fact, of the general Theorem \ref{thm-strength1}) for two commuting actions involving, instead of $K$-syndeticity, a uniformly lower bound on the lower Banach density of the return set. The proof in \cite{j-dcds} relied on certain technical lemmas about the interplay of Hahn-Banach type extensions for finitely additive invariant means and ultralimits of lower Banach densities. Using the stronger $K$-syndeticity formulation, our proof not only establishes a stronger generalization of the result in \cite{j-dcds} to finitely many commuting actions but also significantly simplifies (in particular, no Hahn-Banach type theorems and ultralimits of lower Banach densities are required anymore) and basically follows from an ultralimit construction of Austin's amenable ergodic Szemerédi theorem as stated next.
\end{remark}

A pivotal step in demonstrating Theorem \ref{thm-strength} hinges on employing ultraproducts of measure-preserving dynamical systems. However, the resultant ultraproduct groups are often not countable, and the corresponding Loeb probability spaces lack separability. Addressing these challenges introduces certain measure-theoretic subtleties, discussed comprehensively in \cite{jt-fund}. To navigate around these intricacies, an abstract category of probability algebra dynamical systems $\mathbf{PrbAlg}_\Gamma$ has been identified (see \cite{j-etds,jt-studia,jt-fund,jt-etds}). The abstract system is obtained from a concrete probability space by abstracting away the intrinsic point structure and exclusively focusing on the relationships between measurable sets, considering operations such as intersections, unions, and complementations. $\mathbf{PrbAlg}_\Gamma$, along with the tools to work with its objects, is gathered in Section \ref{sec:category}

Equipped with these tools in the domain of uncountable ergodic theory, we extend Austin's amenable multiple recurrence theorem \cite{austin2016nonconventional} to encompass the actions of uncountable amenable groups on inseparable probability spaces in the following theorem. This uncountable variant of Austin's theorem is then applied to the ultraproduct systems, playing a key role in proving Theorem \ref{thm-strength}.
\begin{theorem}
\label{thm-main1} Let $\Gamma $ be an arbitrary discrete amenable group and let $(X,\mu,T)$ be a $\mathbf{PrbAlg}_{\Gamma^d}$-system, that is, there are  finitely many commuting measure-preserving $\Gamma$-actions $T_i\colon \Gamma\acts(X,\mu)$, $i=1,\ldots,d$, where $(X,\mu)$ is a probability algebra. Let $f_{1},\ldots ,f_{d}\in L^{\infty}(X,\mu)$, and let $(\Phi _{\kappa})$ be a
F\o lner net for $\Gamma$. Then the limit 
\begin{equation}
\lim_{\kappa}\frac{1}{|\Phi _{\kappa}|}\sum_{\gamma \in \Phi
_{\kappa}}\prod_{i=1}^{d}f_i\circ {T_{[1,i]}^{\gamma}}  \label{eq-lim}
\end{equation}
exists in $L^{2}(X,\mu )$ and is independent of the F\o lner net.
Moreover, if a measurable set $E$ in $X$ is such that $\mu (E)>0$, then 
\begin{equation}
\lim_{\kappa}\frac{1}{|\Phi _{\kappa}|}\sum_{\gamma \in \Phi _{\kappa}}\mu \left(
\bigcap_{i=0}^{d}T_{[1,i]}^{\gamma^{-1}}(E)\right) >0.  \label{eq-mr}
\end{equation} 
In particular, there exists $\varepsilon >0$ such that 
\begin{equation}
\left\{ \gamma \in \Gamma \colon \mu \left(
\bigcap_{i=0}^{d}T_{[1,i]}^{\gamma^{-1}}(E)\right) >\varepsilon \right\}
\label{eq-syndetic}
\end{equation} 
is syndetic in $\Gamma $.
\end{theorem}
The case of two commuting transformations of Theorem \ref{thm-main1} was previously established in \cite{j-dcds} by the first author and coauthors. They generalized the proof of the amenable double recurrence theorem by Bergelson, McCutcheon, and Zhang from \cite{bergelson1997roth}. Zorin-Kranich \cite{zorin2016norm} establishes the limit claim \eqref{eq-lim} in Theorem \ref{thm-main1} in full generality using an adaptation of functional analytic methods developed by Walsh \cite{walsh2012norm}, who established the $L^2$-limit in the case of finitely many actions of a nilpotent group. However, Zorin-Kranich's result does not provide information about the limit; in particular, it does not yield the multiple recurrence statements \eqref{eq-mr} and \eqref{eq-syndetic}. These multiple recurrence statements are established in the case of countable amenable groups by Austin \cite{austin2016nonconventional} using sated extensions. Our proof of Theorem \ref{thm-main1} will modify necessary steps in \cite{austin2016nonconventional} to tailor his proof to our uncountable setting. 

\subsection{Combinatorial application}
As an immediate consequence of Theorem \ref{thm-strength}, we obtain the following combinatorial application. For a finite group $G$, we denote by $m$ the uniform measure. On $G^d$, we denote by $m^{\otimes d}$ the $d$-fold product of the uniform measure. 

\begin{corollary}
Let $d,l\geq 1$ be integers and let $\varepsilon > 0$. Then there are $\delta=\delta(d,l,\varepsilon)>0$ and $K=K(d,l,\varepsilon)\geq 1$ such that for all finite solvable groups $G$ of derived length $l$ and any subset $E\subset G^d$ with $m^{\otimes d}(E)\geq \varepsilon$, we have that 
\begin{multline*}
    \{g\in G\colon m^{\otimes d}(\{(a_1,\ldots,a_n)\in G^d\colon \\ (a_1,\ldots,a_n),(ga_1,a_2,\ldots,a_n),\ldots,(ga_1,ga_2,\ldots, ga_n)\in E\})\geq \delta \}
\end{multline*}
is $K$-syndetic. 
\end{corollary}

\begin{remark}
In \cite[Theorem 1.5]{DRZ}, the authors establish $K$-syndeticity for the density of triangle configurations with a Khintchine-type lower bound in a fixed class of quasirandom ultraproduct groups. In \cite[Corollary 1.6]{DRZ}, they deduce a similar consequence for the class of all non-cyclic finite simple groups. Their proof relies on a convergence theorem along minimal idempotent ultrafilters for ergodic averages formed by two commuting actions of a minimally almost periodic group. 
\end{remark}

\subsection{Discussion}

In Theorem \ref{thm-strength1} below, we prove a more general version of Theorem \ref{thm-strength}, establishing a strong form of uniform syndeticity within any \emph{uniformly amenable} class of groups. We then derive Theorem \ref{thm-strength} by demonstrating that the class of solvable groups with a fixed derived length is uniformly amenable. A key property in this verification is that the class of solvable groups with a fixed derived length is closed under countable direct products.

Since Theorem \ref{thm-main1} holds for all amenable groups, a natural question arises: Does Theorem \ref{thm-strength} remain valid for the entire class of amenable groups (that is, $K,\delta$ only depend on $\varepsilon, d$ and are uniform for all amenable acting groups)? The methods employed in this paper cannot address this question\footnote{The countable direct product of arbitrary nilpotent groups is not amenable in general, which essentially undermines our strategy based on Theorem \ref{thm-main1} to prove uniform syndeticity. We are grateful to Dave Benson for this observation. Additionally, an example of an amenable group that is not uniformly amenable is given in \cite{Wysoczanski}}.

On the other hand, the results of \cite{bergelson2007central} give hope that Theorem \ref{thm-main1} holds for all amenable groups (or even for the class of all groups!). In \cite[Theorem 1.3]{bergelson2007central}, the authors establish the consequence\footnote{Actually, they establish the stronger conclusion that the return set is $\mathcal{C}^*$, we refer the interested reader to \cite{bergelson2007central} for the definition of a central$^*$ subset of a group.} \eqref{eq-syndetic} in Theorem \ref{thm-main1} for the action of an arbitrary countable group in the case of two commuting actions. Proving an uncountable version of \cite[Theorem 1.3]{bergelson2007central} should, in principle, yield the analogue of Theorem \ref{thm-strength} for the class of all groups in the case of two commuting actions by the same proof as given in Section \ref{sec-proof1}. To our knowledge, the analogue of \cite[Theorem 1.3]{bergelson2007central} in the case of more than two commuting actions of an arbitrary countable group is unknown\footnote{A potential line of attack, suggested by Austin \cite{austin2016nonconventional}, is to combine the technique of stated extensions and the ultrafilter techniques in \cite{bergelson2007central}}. We hope to address these questions in future work.

\subsection*{Acknowledgments}

AJ was supported by DFG research fellowship JA 2512/3-1. We thank John Griesmer for helpful comments. We are sincerely grateful to an anonymous referee for a very constructive and detailed report, which helped to state a beautiful strengthening of our uniformity result and improve the presentation. 

\section{Tools}\label{sec-prelim}
\subsection{The category of probability algebra dynamical systems }\label{sec:category}
We now formalize the ``point-free'' approach by introducing the category of probability algebra dynamical systems
and the canonical model functor, see Figure \ref{fig:categories}. For a
comprehensive background, references, and any unexplained concepts which are
used in the sequel, we refer the interested reader to \cite{jt-fund}.

\begin{figure}[tbp]
\centering
\begin{tikzcd}
      & \CHProb  \arrow[d,blue, "\Mes\circ \Abs"'] & \CHProb_\Gamma  \arrow[d, blue, "\Mes\circ \Abs"']  \arrow[l, tail, blue]  \\
       & \OpProbAlg \arrow[u, bend right, shift right = 3ex, hook, two heads, "\Stone"'] \arrow[d, "L^2"', blue, tail] & \OpProbAlg_\Gamma \arrow[l, tail, blue]  \arrow[u, bend right, shift right = 3, hook, two heads, "\Stone"']  \arrow[d, "L^2"', blue, tail] \\
      & \Hilb^\op & \Hilb^\op_\Gamma \arrow[l, tail, blue]   \\
    \end{tikzcd}
\caption{The main categories and functors used in this paper ($ \mathrm{op}$
indicates the use of the opposite category). Arrows with tails are faithful
functors and arrows with two heads in one direction are full. Unlabelled
functors are forgetful. The diagram is not fully commutative (even modulo
natural isomorphisms), but the functors in {\color{blue} blue} form a
commuting subdiagram.}
\label{fig:categories}
\end{figure}
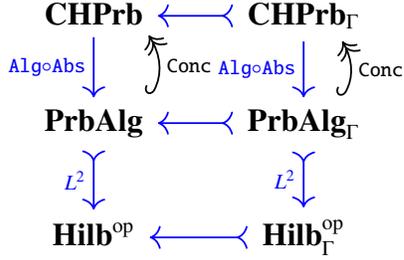

\begin{definition}[Categories \& functors]
\text{}

\begin{itemize}
\item[(i)] We denote by ${\mathbf{CHPrb}}$ the category of compact
Hausdorff spaces equipped with a Baire-Radon probability measure and
measure-preserving continuous functions.

\item[(ii)] We denote by $\mathbf{AbsMbl}$ the category of abstract
measurable spaces which we define as the opposite category of the category
of $\sigma$-complete Boolean algebras and $\sigma$-complete Boolean
homomorphisms.

\item[(iii)] We denote by $\mathbf{PrbAlg}$ the opposite category of the
category of probability algebras and measure-preserving Boolean
homomorphisms. Note that the category $\mathbf{PrbAlg}$ has arbitrary
inverse limits (e.g., see \cite{jt-fund}), a fact which will be
useful for us later.

\item[(iv)] We denote by $\mathbf{Hilb}$ the category of complex Hilbert
spaces and linear isometries.

\item[(v)] Given a (discrete) group $\Gamma$, we can turn a category $ 
\mathcal{C}=\mathbf{Hilb}, \mathbf{PrbAlg}, {\mathbf{CHPrb}} $ into a
dynamical category $\mathcal{C}_\Gamma$ as follows. Given an object $X$ in $ 
\mathcal{C}$, we can associate with $X$ the group $\mathrm{Aut}(X)$ of its
automorphisms in $\mathcal{C}$. The dynamical category $\mathcal{C}_\Gamma$
now consists of pairs $(X,T)$ where $X$ is a $\mathcal{C}$-object and $ 
T\colon \Gamma\to \mathrm{Aut}(X)$ a group homomorphism. A $\mathcal{C}_\Gamma$ 
-morphism is a $\mathcal{C}$-morphism which intertwines with the $\Gamma$ 
-actions.

\item[(vi)] The abstraction functor $\mathtt{Abs}$ maps a concrete
measurable space $(X,\Sigma_X)$ to the $\mathbf{AbsMbl}$-object $\Sigma_X$
and a measurable function $f:(X,\Sigma_X)\to (Y,\Sigma_Y)$ to the $\mathbf{ 
AbsMbl}$-morphism $f^*:\Sigma_X\to \Sigma_Y$ where $f^*$ is the (opposite)
pullback map $f^*(E):=f^{-1}(E)$, $E\in \Sigma_Y$. We apply the $\mathtt{Abs}
$-functor to a concrete probability space $(X,\Sigma_X,\mu)$ and obtain an
abstract probability space $(\Sigma_X,\mu)$. Let $\mathcal{I}_\mu=\{E\in
\Sigma_X: \mu(E)=0\}$ be the ideal of $\mu$-null sets. Then the quotient
Boolean algebra $X_\mu\coloneqq \Sigma_X/\mathcal{I}_\mu$ is $\sigma$ 
-complete. We can lift the measure $\mu$ to $X_\mu$ in a natural way, and by
an abuse of notation, we denote this lift by $\mu$ again. The tuple $ 
(X_\mu,\mu)$ is a probability algebra, and we define $\mathtt{Alg}\circ 
\mathtt{Abs}(X,\Sigma_X,\mu_X):=(X_\mu,\mu)$. If $f\colon
(X,\Sigma_X,\mu)\to (Y,\Sigma_Y,\nu)$ is a measure-preserving function, then
the pullback map $f^*$ maps the ideal $\mathcal{I}_\nu$ to the ideal $ 
\mathcal{I}_\mu$. We obtain a $\mathbf{PrbAlg}$-morphism $\mathtt{Alg}\circ 
\mathtt{Abs}(f)\colon (X_\mu,\mu)\to (Y_\nu,\nu)$.

\item[(vii)] The canonical model functor $\mathtt{Conc}$ reverses the
process described in the previous item. More precisely, if $(X,\mu )$ is a
probability algebra, there exists a ${\mathbf{CHPrb}}$-space $\mathtt{Conc} 
(X,\mu )\coloneqq(Z,\mathcal{B}a(Z),\mu _{Z})$ such that $\mathtt{Alg}\circ 
\mathtt{Abs}(Z,\mathcal{B}a(Z),\mu _{Z})$ is isomorphic to $(X,\mu )$ in $ 
\mathbf{PrbAlg}$. A complete construction of $\mathtt{Conc} 
(X,\mu )$ can be found in \cite{jt-fund}. Given a $\mathbf{PrbAlg}$-morphism $f:(X,\mu )\rightarrow
(Y,\nu )$, we define $\mathtt{Conc}(f):Z_{X}\rightarrow Z_{Y}$ by $\mathtt{ 
Conc}(f)(\theta )=\theta \circ \phi $.

\item[(viii)] Next we define the $L^{2}$-functor. Let $(X,\mu )$ be a $ 
\mathbf{PrbAlg}$-space with canonical model $\mathtt{Conc}(X,\mu )=(Z, 
\mathcal{B}a(Z),\mu _{Z})$ as constructed previously. We define the $L^{2}$ 
-functor on objects by $L^{2}(X,\mu )\coloneqq L^{2}(Z,\mathcal{B}a(Z),\mu
_{Z})$. If $\pi :(X,\mu )\rightarrow (Y,\nu )$ is a $\mathbf{PrbAlg}$ 
-morphism, then $L^{2}(\pi ):L^{2}(Y,\nu )\rightarrow L^{2}(X,\mu )$ is
defined by the Koopman operator $L^{2}(\pi )(f)\coloneqq\pi ^{\ast }f$ where 
$\pi ^{\ast }f\coloneqq f\circ \mathtt{Conc}(\pi )$.

\item[(ix)] Similarly, we define the dynamical version of the functors $ 
\mathtt{Abs}\circ \mathtt{Alg}$, $\mathtt{Conc}$, and $L^{2}$.
\end{itemize}
\end{definition}

A significant feature of the canonical model $\mathtt{Conc}(X)=(Z_X,\mathcal{ 
B}a(Z_X),\mu_{Z_X})$ of a $\mathbf{PrbAlg}$-space $X=(X,\mu)$ is the strong
Lusin property (cf.~\cite[Section 7]{jt-fund}), which states that
the commutative von Neumann algebra $L^\infty(Z_X,\mathcal{B} 
a(Z_X),\mu_{Z_X})$ is isomorphic to the commutative $C^*$-algebra $C(Z_X)$
of continuous functions on $Z_X$ in the category of unital $C^*$-algebras.
In ${\mathbf{CHPrb}}$-spaces with the strong Lusin property every
equivalence class of bounded measurable functions has a continuous
representative. A very useful consequence of this property is that it comes
with a canonical disintegration of measures:

\begin{theorem}
\label{thm-disintegration} Let $\Gamma$ be a discrete group. Let $ 
\pi:(X,\mu,T)\to (Y,\nu,S)$ be a $\mathbf{PrbAlg}_\Gamma$-morphism. Then
there exists a unique Baire--Radon probability measure $\mu_y$ on $Z_{X}$
for each $y \in Z_{Y}$ which depends continuously on $y$ in the vague
topology in the sense that $y \mapsto \int_{Z_{X}} f\ d\mu_y$ is continuous
for every $f$ in the space of continuous functions $C(Z_{X})$, and such that 
\begin{equation}  \label{disint-form}
\int_{Z_{X}} f(x) g(\mathtt{Conc}(\pi)(x))\ d\mu_{Z_{X}}(x) = \int_{Z_{Y}}
\left(\int_{Z_{X}} f\ d\mu_y\right) g\ d\mu_{Z_{Y}}
\end{equation}
for all $f \in C(Z_{X})$, $g \in C(Z_{Y})$. Furthermore, for each $y \in
Z_{Y}$, $\mu_y$ is supported on the compact set $\mathtt{Conc} 
(\pi)^{-1}(\{y\})$, in the sense that $\mu_y(E)=0$ whenever $E$ is a
measurable set disjoint from $\mathtt{Conc}(\pi)^{-1}(\{y\})$. (Note that
this conclusion does \emph{not} require the fibers $\mathtt{Conc} 
(\pi)^{-1}(\{y\})$ to be Baire measurable.) Moreover, we have $ 
\mu_{S^\gamma_{Z_{Y}}(y)}=(T^\gamma_{Z_{X}})^*\mu_y$ for all $y\in Z_{Y}$
and $\gamma\in\Gamma$.
\end{theorem}

Let $\pi:(X,\mu,T)\to (Y,\nu,S)$ be a $\mathbf{PrbAlg}_\Gamma$-morphism and
let $\mathtt{Conc}(\pi):Z_{X}\to Z_{Y}$ be its canonical representation. For
every $f\in L^2(\mathtt{Conc}(Y))$, the composition $\mathtt{Conc}(\pi)^* f$
is an element of $L^2(\mathtt{Conc}(X))$ since $\mathtt{Conc}(\pi)$ is
measure-preserving. In fact, $\{\mathtt{Conc}(\pi)^* f: f\in L^2(\mathtt{Conc 
}(Y))\}$ is a closed $\Gamma$-invariant subspace of $L^2(\mathtt{Conc}(X))$.
Thus we can identify $L^2(\mathtt{Conc}(Y))$ with the closed subspace $ 
\mathtt{Conc}(\pi)^*(L^2(\mathtt{Conc}(Y)))$ in $L^2(\mathtt{Conc}(X))$.
Using this identification, we can define a conditional expectation operator $ 
\mathbb{E}(\cdot|Y)$ from $L^2(\mathtt{Conc}(X))$ to $L^2(\mathtt{Conc}(Y))$ 
. Since $L^\infty$ is dense in $L^2$ in the $L^2$ topology and by Theorem  
\ref{thm-disintegration}, we obtain the disintegration of measures 
\begin{equation*}
\mathbb{E}(f|Y)(y)=\int_{Z_{X}} f d\mu_y
\end{equation*}
almost surely for all $f\in L^2(\mathtt{Conc}(X))$.

An important application is a canonical construction of relatively
independent products. Indeed, let $\pi _{1}:(X_{1},\mu
_{1},T_{1})\rightarrow (Y,\nu ,S)$, $\pi _{2}:(X_{2},\mu
_{2},T_{2})\rightarrow (Y,\nu ,S)$ be two $\mathbf{PrbAlg}_{\Gamma }$ 
-morphisms. Let $(\mu _{y}^{1})_{y\in Z_{Y}}$ and $(\mu _{y}^{2})_{y\in
Z_{Y}}$ be the corresponding canonical disintegration of measures. Define
the probability measure 
\begin{equation*}
\mu _{Z_{X_{1}}}\times _{Z_{Y}}\mu _{Z_{X_{2}}}(E):=\int_{Z_{Y}}\mu
_{y}^{1}\times \mu _{y}^{2}(E)d\mu _{Z_{Y}}
\end{equation*} 
for all $E\in \mathcal{B}a(Z_{X_{1}})\otimes \mathcal{B}a(Z_{X_{2}})$. Then 
\begin{equation*}
(Z_{X_{1}}\times Z_{X_{2}},\mathcal{B}a(Z_{X_{1}})\otimes \mathcal{B} 
a(Z_{X_{2}}),\mu _{Z_{X_{1}}}\times _{Z_{Y}}\mu
_{Z_{X_{2}}},T_{Z_{X_{1}}}\times T_{Z_{X_{2}}})
\end{equation*} 
is a ${\mathbf{CHPrb}}_{\Gamma }$-object coming with two ${\mathbf{CHPrb}} 
_{\Gamma }$-morphisms $\psi _{1}:Z_{X_{1}}\times Z_{X_{2}}\rightarrow
Z_{X_{1}}$ and $\psi _{2}:Z_{X_{1}}\times Z_{X_{2}}\rightarrow Z_{X_{2}}$.
Applying the functor $\mathtt{Alg}\circ \mathtt{Abs}$, we obtain a $\mathbf{ 
PrbAlg}_{\Gamma }$-object $(X_{1}\times _{Y}X_{2},\mu _{1}\times _{Y}\mu
_{2},T_{1}\times T_{2})$ and the two $\mathbf{PrbAlg}_{\Gamma }$-morphisms $ 
\mathtt{Alg}\circ \mathtt{Abs}(\psi _{1}),\mathtt{Alg}\circ \mathtt{Abs} 
(\psi _{2})$ satisfying the following commutative diagram in $\mathbf{PrbAlg} 
_{\Gamma }$: 
\begin{equation*}
\begin{tikzcd} & {X_1\times_Y X_2} \\ {X_1} && {X_2} \\ & {Y}
\arrow["{\mathtt{Alg}\circ\mathtt{Abs}(\psi_1)}" description, from=1-2,
to=2-1] \arrow["{\mathtt{Alg}\circ \mathtt{Abs}(\psi_2)}" description,
from=1-2, to=2-3] \arrow["{\pi_1}" description, from=2-1, to=3-2]
\arrow["{\pi_2}" description, from=2-3, to=3-2] \end{tikzcd}
\end{equation*}
\subsection{Ultrafilters and nonstandard analysis}
A filter on a set $X$ is a non-empty collection $f$ of subsets of $X$ satisfying the following properties.
\begin{itemize}
    \item[(i)] $\emptyset \notin f$.
    \item[(ii)] If $A,B\in f$, then $A\cap B\in f$.
    \item[(iii)] If $A\in f$, $B\subset X$, and $A\subset B$, then $B\in f$.
\end{itemize}
An ultrafilter on $X$ is a maximal element in the set of filters on $X$ with respect to set inclusion. A \emph{non-principal ultrafilter} is an ultrafilter such that none of its elements is finite.

We are concerned with non-principal ultrafilters on the set of natural numbers $\mathbb{N}$.
The \emph{Fr\'echet filter} consists of all subsets $A\subset \mathbb{N}$ for which there is $n\in \mathbb{N}$ such that $A$ contains the tail $\{n,n+1,\ldots\}$. By definition, the Fr\'echet filter is contained in any non-principal ultrafilter on $\mathbb{N}$.

Fix a non-principal ultrafilter $p$ on $\mathbb{N}$. The \emph{ultraproduct} of a sequence $\{X_n\}$ of sets with respect to $p$ is the quotient set $X^*=\prod_{n\to p} X_n$ of the Cartesian product $\prod_{n\in \mathbb{N}} X_n$ with respect to the equivalence relation $(x_n)\sim (y_n)$ if and only if $\{n\colon x_n=y_n\}\in p$. The equivalence class of $(x_n)$ is denoted by $\lim_{n\to p} x_n$, and called an \emph{ultralimit} of the elements $x_n$. 

A subset $A$ of $X^*$ is said to be \emph{internal} if it is of the form $A=\prod_{n\to p} A_n$ for some subset $A_n\subset X_n$ for each $n$. One can check that the collection of internal subsets of $X^*$ forms an algebra of sets, in particular $\bigcup_{i=1}^K A_i= \prod_{n\to p} \bigcup_{i=1}^K A_{i,n}$ for finitely many internal subsets $A_i$ of $X^*$. 

One can also verify that the ultraproduct of a sequence of groups is a group. Moreover, a class $\mathfrak{G}$ of groups is said to be \emph{uniformly amenable} if the ultraproduct group $\Gamma^*$ is amenable for any sequence of groups $\Gamma_n\in\mathfrak{G}$. 

Let $(r_n)$ be a bounded sequence of real numbers. By the Bolzano--Weierstrass theorem, there is a unique real number $r$ such that $\lim_{n\to p} (r_n - r)$ is infinitesimal, that is an ultralimit real number in the equivalence class of a null sequence. We define the standard part function $\mathrm{st}(\lim_{n\to p}r_n)\coloneqq r$.

\section{Proof of Theorem \ref{thm-strength}}\label{sec-proof1}

We prove Theorem \ref{thm-strength} via an ultralimit construction from Theorem \ref{thm-main1}. The following simple lemma is key. Throughout, fix a non-principal ultrafilter $p$ on $\mathbb{N}$. 
\begin{lemma}\label{k-synd}
Let $\{\Gamma_n\}$ be a sequence of groups and let $\Gamma^*=\prod_{n\to p}\Gamma_n$ be their ultraproduct.  
Let $S=\prod_{n\to p} S_n$ be an internal subset of $\Gamma^*$ which is a $K$-syndetic subset of $\Gamma^*$ for some $K\geq 1$.
Then $$\{n\colon S_n \text{ is a $K$-syndetic set in $\Gamma_n$} \}\in p.$$  
\end{lemma}

\begin{proof}
    There are $\gamma_1,\ldots,\gamma_K\in \Gamma$ such that $\Gamma^*=\bigcup_{i=1}^K \gamma_i \cdot S$. 
    We have $\gamma_i=\lim_{n\to p} \gamma_{i,n}$ for some choice of $\gamma_{i,n}\in \Gamma_n$ for each $n$ and every $1\leq i\leq K$ such that  
    $$\Gamma^* = \bigcup_{i=1}^K \gamma_i\cdot S = \prod_{n\to p} \bigcup_{i=1}^K \gamma_{i,n}\cdot S_n,$$
    and thus $\{n\colon \bigcup_{i=1}^K \gamma_{i,n}\cdot S_n=\Gamma_n\}\in p$. 
\end{proof}

At the end of this section, we show how to derive Theorem \ref{thm-strength} from the following generalization of it. 

\begin{theorem}\label{thm-strength1}
For every uniformly amenable class $\mathfrak{G}$ of groups, for all integer $d\geq 1$, and any $\varepsilon>0$ there exist an integer $K\geq 1$ and $\delta>0$ (only
depending on $\mathfrak{G},\varepsilon,d$) such that for any group $\Gamma\in\mathfrak{G}$, any probability space $(X,\mu)$, every finitely many pairwise commuting measure-preserving actions $T_i\colon \Gamma\acts(X,\mu)$, $i=1,\ldots,d$, and every measurable set $E$ in $X$ with $\mu(E)\geq \varepsilon$ it holds that 
\begin{equation*}
\left\{ \gamma \in \Gamma \colon \mu
(T_{1}^{\gamma ^{-1}}(E)\cap \left( T_{[1,2]}^{\gamma }\right) ^{-1}(E)\cap
\cdots \cap \left( T_{[1,d]}^{\gamma }\right)^{-1}(E))\geq \delta \right\}
\end{equation*}
is $K$-syndetic, where $T_{[a,b]}^{\gamma }:=T_{a}^{\gamma }\circ T_{a+1}^{\gamma }\circ \cdots \circ T_{b}^{\gamma}$.
\end{theorem}

\begin{proof}
Towards a contradiction, assume there is a uniformly amenable class of groups $\mathfrak{G}$, $d\geq 1, \varepsilon >0$ such that for every $n\geq 1$ there is a group $\Gamma_n\in\mathfrak{G}$, $d$ pairwise commuting measure-preserving $\Gamma_n$-actions $T_{n,1},\ldots, T_{n,d}$ on a probability space $(X_n, \mathcal{X}_n,\mu_n)$ and $E_{n}\in \mathcal{X}_n$ with $\mu _{n}(E_{n})\geq \varepsilon $ such that 
\begin{equation}
A_n\coloneqq \left\{ \gamma \in \Gamma_n \colon \mu
(T_{n,1}^{\gamma ^{-1}}(E_{n})\cap \left( T_{n,[1,2]}^{\gamma }\right)
^{-1}(E_{n})\cap \cdots \cap \left( T_{n,[1,d]}^{\gamma }\right)
^{-1}(E_{n})\geq 1/n\right\}\label{eq-needlater}
\end{equation} 
is not $n$-syndetic.

Let $X^*=\prod_{n\to p} X_n$ be the ultraproduct of the sets $X_n$ and denote by 
\begin{equation*}
\mathcal{A}=\left\{\prod_{n\to p} D_n: (D_n)\in \prod_{n\in\mathbb{N}} 
\mathcal{X}_n\right\},
\end{equation*}
the algebra of internal subsets of $X^*$. We define the Loeb premeasure 
\begin{equation*}
\mu^*:\mathcal{A}\to [0,1], \quad \mu^*\left(\prod_{n\to p} D_n\right):= 
\mathrm{st}\left(\lim_{n\to p} \mu_n(D_n)\right).
\end{equation*}
By Carath\'eodory's extension and uniqueness theorem, $\mu^*$ extends to a unique countably additive probability measure $\mu$ on the $\sigma$-algebra of sets generated by $\mathcal{A}$. By a slight abuse of notation, we let $(X_\mu,\mu)$ denote the probability algebra associated to $(X^*,\sigma(\mathcal{A}),\mu)$.

We let $\Gamma^*=\prod_{n\rightarrow p}\Gamma_n$. Since $\mathfrak{G}$ is uniformly amenable, $\Gamma^*$ is an amenable group. For each $i=1,\ldots,d$, and for every $\gamma^*=\lim_{n\to p} \gamma_n \in
\Gamma^*$ and $\prod_{n\to p} D_n \in\mathcal{A}$, define 
\begin{equation*}
(T^*_i)^{\gamma^*}\left(\prod_{n\to p} D_n\right)\coloneqq  
\prod_{n\to p} T^{\gamma_n}_{n,i}(D_n)
\end{equation*}
One checks that $(T^*_i)^{\gamma^*}$ is a well-defined $\Gamma^*$-action by Boolean automorphism of $\mathcal{A}$ which preserves the probability measure $\mu$. Since $\mu$ is a finite measure, by a standard approximation result in measure theory (see, e.g., \cite[Theorem 5.7]{bauer-measure}), for any $D\in X_\mu$ there is a sequence $D_n\in\mathcal{A}$ such that $\mu(D\Delta D_n)\to 0$ as $n$ tends to $\infty$, where $\Delta$ denotes symmetric set difference. Thus, we
can extend actions $T^*_i$ to abstract $\mathbf{PrbAlg}_\Gamma$-actions 
$T_i:\Gamma^*\to \mathrm{Aut}(X_\mu,\mu)$. We obtain an abstract $(\Gamma^*)^d$-system $(X_\mu,\mu,T)$.

By construction, we have $\mu (E^{\ast })\geq \varepsilon $ where $E^{\ast
}:=\prod_{n\rightarrow p}E_{n}$ and the $E_{n}$ are as in \eqref{eq-needlater} 
. By Theorem \ref{thm-main1}, there exist $\delta >0$ and $K\geq 1$ such that 
$$
B = \left\{ \gamma \in \Gamma ^{\ast }\colon \mu
((T_{1}^{\ast })^{\gamma ^{-1}}(E^{\ast })\cap (T_{[1,2]}^{\ast })^{\gamma
^{-1}}(E^{\ast })\cap \cdots \cap (T_{[1,d]}^{\ast })^{\gamma ^{-1}}(E^{\ast
}))\geq \delta \right\}$$ 
is $K$-syndetic. 
Let 
\begin{equation*}
B_{n}=\{\gamma \in \Gamma_n \colon \mu _{n}(T_{n,1}^{\gamma ^{-1}}(E_{n})\cap
(T_{n,[1,2]}^{\gamma })^{^{-1}}(E_{n})\cap \cdots \cap (T_{n,[1,d]}^{\gamma
})^{^{-1}}(E_{n}))\geq \delta \}.
\end{equation*} 
By construction, we have $B=\prod_{n\to p} B_n$.  
Then Lemma \ref{k-synd} gives $$\{n \colon B_n \text{ is $K$-syndetic in $\Gamma_n$}\}\in p$$
Moreover, $B_n\subset A_n$ as long as $n>\frac{1}{\delta}$, and since the Fr\'{e}chet filter is
contained in any non-principal ultrafilter, we must have $\{n\colon B_n\subset A_n\}\in p$. 
Since a filter is intersection closed and $p$ does not contain finite sets as a non-principal ultrafilter, there are infinitely many $A_n$ which are $K$-syndetic, contradicting our assumptions.  
\end{proof}
As a corollary, we obtain Theorem \ref{thm-strength}: 
\begin{proof}
By Theorem \ref{thm-strength1}, it suffices to verify that the class of solvable groups of derived length at most $l$ is uniformly amenable for some fixed $l\geq 1$. 

Let $\Gamma^*$ be the ultraproduct of a sequence $\{\Gamma_n\}$ of solvable groups of derived length at most $l$. We claim that $\Gamma^*$ is solvable (and amenable in particular). The direct product $\prod_{n}\Gamma_n$ is a solvable group as $\{\Gamma_n\}$ has uniformly bounded length. The ultraproduct $\Gamma^*=\prod_{n\rightarrow p}\Gamma_n$ is a quotient group of $\prod_{n}\Gamma_n$ so $\Gamma^*$ is solvable as well.  
\end{proof}

\section{Proof of Theorem \ref{thm-main1}}\label{sec-austin}

Austin \cite{austin2016nonconventional} proved a subcase of Theorem \ref{thm-main1} when the acting group $\Gamma$ is countable and the space $(X,\mu)$ is standard Lebesgue by constructing characteristic spaces on stated extensions. Our proof of Theorem \ref{thm-main1} aims to facilitate Austin's proof \cite{austin2016nonconventional} to be carried out in a setup where spaces may not be standard Lebesgue and groups may not be countable. We will reuse most of the arguments in Austin \cite{austin2016nonconventional} and only modify the steps in which the assumptions about the space and the group are substantially used. We will now utilize the notation introduced in Section \ref{sec:category}.

Throughout this section, we let $\Gamma$ be a (discrete) amenable group and $(\Phi_\kappa)$ be a left Følner net for $\Gamma$. Let $(X,\mu)$ be a $\OpProbAlg$-space and $T_{1},T_{2},\ldots ,T_{d}$ be commuting group homomorphisms $T_{i}\colon \Gamma \rightarrow \mathrm{Aut}(X,\mu)$, that is $T_{i}^{\gamma }\circ T_{j}^{\eta }=T_{j}^{\eta }\circ T_{i}^{\gamma }$ for all $\gamma ,\eta \in \Gamma$ and $1\leq i<j\leq d$ (the automorphism group $\mathrm{Aut}(X,\mu)$ is taken in the category $\mathbf{PrbAlg}$). We denote by $(\tilde{X},\tilde{\mu},\tilde{T})$ the corresponding canonical model, so that $\tilde{X}$ is a compact Hausdorff space and each $\tilde{T}_{i}$ acts on $(\tilde{X},\tilde{\mu})$ by measure-preserving homeomorphisms. This canonical model has the advantage that tools such as disintegration of measure and relative independent product are available as discussed in Section \ref{sec:category}. For the most part, we can then adopt Austin's arguments on the concrete space $(\tilde{X},\tilde{\mu},\tilde{T})$. Then the corresponding results for the probability algebra $(X,\mu)$ follow by applying the functor $\mathtt{Alg}\circ \mathtt{Abs}$.

We begin with Zorin-Kranich's convergence theorem \cite{zorin2016norm} which
holds in the generality that we have just set up. 

\begin{theorem}[Zorin-Kranich's convergence theorem]
\label{thm-convergence} Let $f_{1},f_{2},\ldots ,f_{d}\in L^{\infty }(X,\mu )
$. Then the limit 
\begin{equation*}
\lim_{\kappa}\frac{1}{|\Phi _{\kappa}|}\sum_{\gamma \in \Phi
_{\kappa}}\prod_{i=1}^{d}f_{i}\circ T_{[1,i]}^{\gamma }
\end{equation*}%
exists in $L^{2}(X,\mu )$ and is independent of the choice of the (left) F\o %
lner net $(\Phi _{\kappa})$.
\end{theorem}
This is exactly claim \eqref{eq-lim} in Theorem \ref{thm-main1}.
It remains to prove claim \eqref{eq-mr} in the same theorem. This is
achieved as follows. Consider the set of $d$-fold couplings on $\tilde{X}^{d}$, which is the collection of Baire probability measures on $(\tilde{X}^{d},\mathcal{B}a(\tilde{X}^{d}))$ all of whose coordinate projections are $\tilde{\mu}$. By Theorem \ref{thm-convergence}, for each $x\in \tilde{X}$, the averages 
\begin{equation*}
\frac{1}{|\Phi _{\kappa}|}\sum_{\gamma \in \Phi _{\kappa}}\delta _{\tilde{T}%
_{1}^{\gamma }x,\tilde{T}_{[1,2]}^{\gamma }x,\ldots ,\tilde{T}%
_{[1,d]}^{\gamma }x}
\end{equation*}%
converge weakly to a measure $\lambda$ in the set of $d$-fold couplings. This weak convergence implies that 
\begin{equation*}
\frac{1}{|\Phi _{\kappa}|}\sum_{\gamma \in \Phi _{\kappa}}\tilde{\mu}\left(\tilde{T}%
_{1}^{\gamma ^{-1}}(E)\cap \tilde{T}_{[1,2]}^{\gamma ^{-1}}(E)\cap \ldots
\cap \tilde{T}_{[1,d]}^{\gamma ^{-1}}(E)\right)\rightarrow \lambda (E^{d})
\end{equation*}%
for all $E\in \mathcal{B}a(\tilde{X})$. To establish multiple
recurrence, it suffices to show that 
\begin{equation}
\lambda (E_{1}\times E_{2}\times \cdots \times E_{d})=0\Rightarrow 
\tilde{\mu}(E_{1}\cap E_{2}\cap \ldots \cap E_{d})=0.  \label{removal lemma}
\end{equation}
Austin first introduces the notion of satedness and proves that any space has an extension that is sated. It's then not a loss of generality to assume that $(\tilde{X},\tilde{\mu})$ is a sated space in the first place. The advantage of a sated space is that it constrains how a certain relevant $\sigma $-subalgebra of the space is lifted to any of its extensions. More about satedness can be found in Section \ref{sec:sated}.

Define 
\begin{eqnarray*}
H_{i,j} &:&=\{\mathbf{\gamma}\in \Gamma^{d}:\gamma_{i+1}=\gamma_{i+2}=\cdots =\gamma_{j}\} \\
L_{i,j} &:&=\{\mathbf{\gamma}\in H_{i,j}:\gamma_{l}=1\text{ for all }l\notin
(i,j]\}
\end{eqnarray*}
Austin then constructs recursively a tower of $\Gamma ^{d}$-spaces, as a variant
of Host--Kra self-joinings \cite{host-commuting,host2005nonconventional}: 
\[
(\tilde{Y}^{(d)},\tilde{\nu}^{(d)},S^{(d)})\rightarrow (\tilde{Y}^{(d-1)},%
\tilde{\nu}^{(d-1)},\tilde{S}^{(d-1)})\rightarrow \cdots \rightarrow (\tilde{%
Y}^{(0)},\tilde{\nu}^{(0)},\tilde{S}^{(0)})=(\tilde{X},\tilde{\mu},\tilde{T}%
).
\]%
Assume that the tower has already been constructed up to some $j\leq d-1$.
Define an $H_{d-j-1,d}$-action $\tilde{R}^{(j)}$ on $(\tilde{Y}^{(j)},%
\tilde{\nu}^{(j)})$ by setting%
\[
\tilde{R}_{i}^{(j)}=\left\{ 
\begin{array}{cc}
\tilde{S}_{i}^{(j)} & \text{for }i<d-j-1 \\ 
\tilde{S}_{[d-j-1,d]}^{(j)} & i=d-j-1 \\ 
\text{id} & i=[d-j,d-1]%
\end{array}%
\right. .
\]
An $H_{d-j-1,d}$-space is then constructed by the relative product 
\[
\left( \tilde{Z}^{(j+1)},\tilde{\theta}^{(j+1)},\tilde{R}^{(j+1)}\right) =(%
\tilde{Y}^{(j)}\times \tilde{Y}^{(j)},\tilde{\nu}^{(j)}\otimes _{\Sigma _{%
\tilde{Y}^{(j)}}^{L_{d-j-1,d}}}\tilde{\nu}^{(j)},(\tilde{S}%
^{(j)})_{ H_{d-j-1,d}}\times \tilde{R}^{(j)}),
\]
where for a subgroup $H$ of $\Gamma^d$, we let $(\tilde{S})_H$ be the restriction of an action $\tilde{S}$ by $\Gamma^d$ to an action by $H$, and for a subgroup $L$ of $\Gamma^d$, we denote by $\Sigma_{\tilde{Y}}^L$ the $\sigma$-subalgebra of $\Sigma_{\tilde Y}$ of invariant sets with respect to the restriction of an action $\tilde{S}$ by $\Gamma^d$ to an action by $L$.

Finally, a \textit{lifting lemma} proves the existence of a $\Gamma^d$-space extension 
$$\pi:(\tilde{Y}^{(j+1)},\tilde{%
\nu}^{(j+1)},\tilde{S}^{(j+1)})\rightarrow (\tilde{Y}^{(j)},\tilde{%
\nu}^{(j)},\tilde{S}^{(j)}) $$ which admits a commutative diagram of $H_{d-j-1,d}$-spaces
\begin{equation*}
\begin{tikzcd} & \tilde{Z}^{(j+1)} \arrow{dr}{} \\ \left(\tilde{Y}^{(j+1)}\right)_{ H_{\{d-j-1,d\}}}\arrow{ur}{}
\arrow{rr}{\pi} && \left(\tilde{Y}^{(j)}\right)_{ H_{\{d-j-1,d\}}} \end{tikzcd},
\end{equation*}
where $Y_H$ is the $H$-space with the same probability space but with the action restricted to $H$.

The advantage of the sequence of extensions is a variant of Host--Kra inequality: the asymptotic behavior of the ergodic average of $f_{i}$s (the term inside the limit of (\ref{eq-lim})) is governed by an integral of a product of those $f_{i}$s lifted to $\tilde{Y}^{(d)}$.

These Host--Kra-like self-joinings admit characteristic subspaces. 
A closed subspace $V\leq L^{2}(\tilde{\mu})$ is \textbf{partially
characteristic in position} $i$ if the ergodic averages of $f_{1},\ldots,f_{d}$
are asymptotically the same as those of $%
f_{1},\ldots,f_{i-1},P^{V}f_{i},f_{i+1},\ldots,f_{d}$, where $P^V$ denotes the orthogonal projection onto the space $V$. In a sated space $(\tilde{X},\tilde{\mu})$, it can be proven that 
\[
L^{2}\left( \tilde{\mu} \bigg|\bigvee\limits_{l=0}^{i-1}\Sigma^{\tilde{T}_{(l;i]}}_{\tilde{X}}\vee
\bigvee\limits_{l=i+1}^{d}\Sigma ^{\tilde{T}_{(i;l]}}_{\tilde{X}}\right) 
\]%
is partially characteristic in position $i$. The significance of these characteristic subspaces is that $P^Vf_d$ can be approximated by a finite sum of products of the form $h_0h_1\ldots h_{d-1}$, where each $h_i$ is $\Sigma^{\tilde{T}_{(i;d]}}_{\tilde{X}}$-measurable. This then allows us to reduce an ergodic average of $d$ functions to an ergodic average of $d-1$ functions.

Satedness helps to prove that some spaces related to the characteristic subspaces are relatively orthogonal. As an illustrative example, let's say we want to prove that $L^{2}(\tilde{\mu}|\Phi _{1})$ and $L^{2}(\tilde{\mu}|\Phi _{2})$ are relatively independent over $V(\tilde{X})$, where $\Phi _{1},\Phi _{2}$ are $\sigma $-algebras over $\tilde{X}$ and $V(\cdot)$ is a functorial $L^2$-subspace ($V(\tilde{Z})$ is an $L^2$-subspace of $L^2(\tilde{Z})$; see Definition \ref{def:sated} for details). We assume that $\tilde X$ is a $\Psi $-sated space.

Let $f\in L^{2}(\tilde{\mu}|\Phi _{1})$ and $g\in L^{2}(\tilde{\mu}|\Phi _{2})$. We construct a relative product measure 
\[
\left( \tilde{Y},\tilde{\nu}\right) =(\tilde{X}^{2},\tilde{\mu}\otimes_{\Phi _{1}}\tilde{\mu})
\]
and carefully define a $\Gamma $-action on the space. Let $\beta_1$ and $\beta_2$ be the projections of $\tilde{Y}=\tilde{X}^2$ onto the first and second coordinate respectively. 
Since $f$ is $\Phi _{1}$-measurable, we have
\begin{equation}
\int_{\tilde{X}}fgd\tilde{\mu}=\int_{\tilde{Y}}(f\circ \beta _{2})(g\circ \beta _{2})d\tilde{\nu}=\int_{\tilde{Y}}(f\circ \beta _{1})(g\circ \beta _{2})d\tilde{\nu}  \label{eqn:relprod}
\end{equation}

Then we use $V$-satedness of $\tilde{X}$, which gives that $L^{2}(\tilde{\mu})\circ \beta _{1}$ and $V(\tilde{Y})$ are relatively orthogonal over $V(\tilde{X})\circ \beta _{1}$. We will need $g\circ \beta _{2}\in V(\tilde{Y})$, so (\ref{eqn:relprod}) equals
\[
\int_{\tilde{Y}}(P^{V(\tilde{X})}f\circ \beta _{1})\cdot (g\circ \beta _{2})d\tilde{\nu}
\]
By the same line of reasoning as in (\ref{eqn:relprod}), we have this equals
\[
\int_{\tilde{X}}P^{V(\tilde{X})}f\cdot gd\tilde{\mu}=\int_{\tilde{X}}P^{V(\tilde{X})}f\cdot P^{V(\tilde{X})}gd\tilde{\mu}
\]
as desired. 

From here, Austin's ergodic version of Tao's Removal Lemma \cite{tao-removal} yields \eqref{removal lemma}.

The modifications we need to extend these arguments to our uncountable setup are:

\begin{itemize}
\item[(i)] Construction of sated extensions for $\OpProbAlg$-spaces.

\item[(ii)] Lifting lemma: extending a factor map relative to a subgroup to the whole group for $\OpProbAlg$-spaces and for uncountable groups.
\end{itemize}

These modifications are carried out in the following two subsections.

\subsection{$\mathbf{PrbAlg}_\Gamma$-sated extensions}\label{sec:sated}

In this section, we verify that probability algebra dynamical systems admit
sated extensions. 

Following the standard notation, if $\pi:(Y,\nu,S)\rightarrow (X,\mu,T)$ is an extension, we let ${\pi}^* f:=f\circ \pi$ on $L^2(X)$.  If $\mathcal{H}$ is a closed subspace of a Hilbert space, we denote by $P_\mathcal{H}$ the orthogonal
projection onto $\mathcal{H}$. 

Recall that if $\mathcal{H}_1,\mathcal{H}_2$, and $\mathcal{I}$ are closed
subspaces of a Hilbert space, then $\mathcal{H}_1$ and $\mathcal{H}_2$ are
said to be \emph{relatively orthogonal} over $\mathcal{I}$ if for any $u\in 
\mathcal{H}_1$ and $v\in \mathcal{H}_2$, we have $\left\langle
u,v\right\rangle =\left\langle P_{\mathcal{I}}u, P_{\mathcal{I} 
}v\right\rangle$. The following simple characterization of relative orthogonality will be useful.

\begin{lemma}
\label{ortho} Suppose $\mathcal{H}_1,\mathcal{H}_2$, and $\mathcal{I}$ are
closed subspaces of a Hilbert space. If in addition $\mathcal{I}\subset 
\mathcal{H}_2$, then $\mathcal{H}_1$ and $\mathcal{H}_2$ are relatively
orthogonal over $\mathcal{I}$ if and only if for any $u\in \mathcal{H}_1$, $ 
P_\mathcal{I}(u)=P_{\mathcal{H}_2}(u)$.
\end{lemma}

\begin{proof}
Suppose $\mathcal{H}_1$ and $\mathcal{H}_2$ are relatively orthogonal over $ 
\mathcal{I}$. Fix $u\in \mathcal{H}_1$. For any $v\in \mathcal{H}_2$, $ 
\left\langle u,v\right\rangle =\left\langle P_{\mathcal{I}}u, P_{\mathcal{I} 
}v\right\rangle=\left\langle P_{\mathcal{I}}u, v\right\rangle$. Since $P_{ 
\mathcal{I}}u\in \mathcal{H}_2$, we have $P_{\mathcal{H}_2}u=P_{\mathcal{H} 
_2}(P_{\mathcal{I}}u)=P_{\mathcal{I}}u$.

Conversely, suppose for any vector in $\mathcal{H}_1$, its projection onto $ 
\mathcal{I}$ and ${\mathcal{H}_2}$ are the same. For any $u\in \mathcal{H}_1$
and $v\in \mathcal{H}_2$, we have 
\begin{equation*}
\left\langle P_{\mathcal{I}}u, P_{\mathcal{I}}v\right\rangle=\left\langle P_{ 
\mathcal{I}}u, v\right\rangle=\left\langle P_{\mathcal{H}_2}u,
v\right\rangle=\left\langle u, v\right\rangle
\end{equation*}
as desired. 
\end{proof}

\begin{definition}[$\mathbf{PrbAlg}_\Gamma$-sated extensions]\label{def:sated}

A \emph{functorial} $L^2$-\emph{subspace} of $\mathbf{PrbAlg} 
_\Gamma$-spaces is a composition of functors $V=W\circ L^2$, where $W$ is a
functor from the category $\mathbf{Hilb}_\Gamma$ to $\mathbf{Hilb}$ sending
any object $\mathcal{H}$ to a closed subspace $V(\mathcal{H})$ of $\mathcal{H 
}$ and any morphism $\phi$ from a $\mathbf{Hilb}_\Gamma$-object $\mathcal{H}$
to a $\mathbf{Hilb}_\Gamma$-object $\mathcal{K}$ to the restriction $ 
V(\phi):V(\mathcal{H})\to V(\mathcal{K})$.

Let $V$ be a functorial $L^2$-subspace of $\mathbf{PrbAlg}_\Gamma$-spaces. A 
$\mathbf{PrbAlg}_\Gamma$-space $X=(X,\mu)$ is said to be $V$-\emph{sated} if
for any $\mathbf{PrbAlg}_\Gamma$-morphism $\pi:(Y,\nu)\rightarrow (X,\mu)$,
the subspaces $\pi^*(L^2(X))$ and $V(Y)$ of $L^2(Y)$ are relatively
orthogonal over their common further subspace $\pi^*(V(X))$. The condition is equivalent to $\pi^*(P_{V(X
)}(h))=P_{V(Y)}(\pi^*(h))$ for any $h\in L^2(X )$, by Lemma \ref{ortho} and the inclusion relation $\pi^*(V(X )) \subset V(Y)$. Moreover, a $\mathbf{ 
PrbAlg}_\Gamma$-morphism $\pi:(Y,\nu)\rightarrow (X,\mu)$ is said to be 
\emph{relatively} $V$-\emph{sated} if for any further $\mathbf{PrbAlg} 
_\Gamma $-morphism $\psi: (Z,\lambda)\rightarrow (Y,\nu)$, the subspaces $ 
(\pi\circ \psi)^*(L^2(X))$ and $V(Z)$ are relatively orthogonal over $ 
\psi^*(V(Y))$.
\end{definition}

For the remainder of this section, we fix a functorial $L^2$-subspace $V$.

\begin{figure}[tbp]
\centering
\begin{tikzcd}
                & L^2(W) \arrow[rd] \arrow[ld] &                             &                             &                  \\
V(W) \arrow[rd] &                             & L^2(Z) \arrow[ld] \arrow[rd] &                             &                  \\
                & V(Z) \arrow[rd]             &                             & L^2(Y) \arrow[ld] \arrow[rd] &                  \\
                &                             & V(Y) \arrow[rd]             &                             & L^2(X ) \arrow[ld] \\
                &                             &                             & V(X )                        &                 
\end{tikzcd}
\caption{The subspace relations in Lemma \protect\ref{relative sated}.}
\label{fig:rel-sated}
\end{figure}
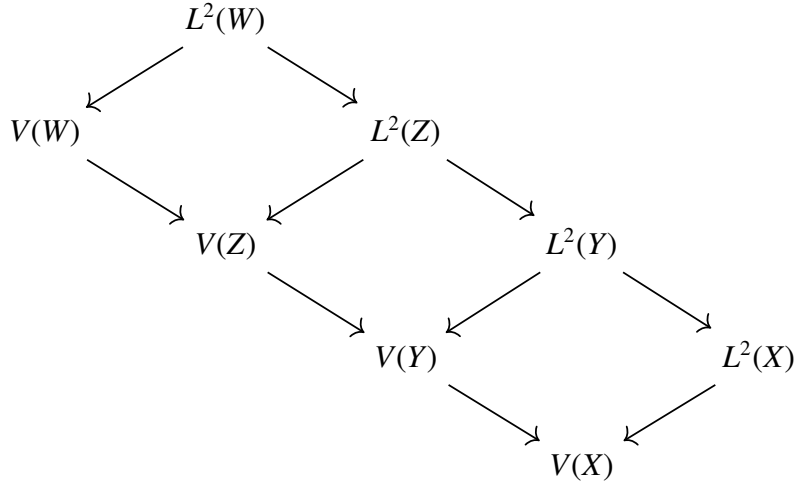

\begin{lemma}
\label{relative sated} Suppose that $\pi:Y\rightarrow X $ is a relatively $V$-sated $\mathbf{PrbAlg}_\Gamma$-morphism and $\phi:Z\rightarrow Y$ is a $ 
\mathbf{PrbAlg}_\Gamma$-morphism. Then $\pi\circ \phi: Z\rightarrow X $ is
relatively $V$-sated.
\end{lemma}

\begin{proof}
Let $\psi:W\rightarrow Z$ be a $\mathbf{PrbAlg}_\Gamma$-morphism. Fix $f\in
L^2(X )$ and $g\in V(W)$. We want to show 
\begin{equation*}
\left\langle (\pi \circ \phi \circ \psi)^*f ,g\right\rangle_{L^2(W)}
=\left\langle \psi^*P_{V(Z)}((\pi \circ \phi )^*f) ,g\right\rangle_{L^2(W)},
\end{equation*} 
(see Figure \ref{fig:rel-sated}). Applying the definition of relative
satedness of $Y\overset{\pi }{\rightarrow }X $ to the further extension $W 
\overset{\phi \circ \psi }{\rightarrow }Y$, we have 
\begin{equation*}
\left\langle (\pi \circ \phi \circ \psi)^*f,g\right\rangle_{L^2(W)}
=\left\langle (\phi\circ \psi)^*P_{V(Y)}(\pi^*f),g\right\rangle_{L^2(W)}.
\end{equation*} 
It remains to prove $P_{V(Z)}((\pi \circ \phi)^*f)=\phi^*P_{V(Y)}(\pi^*f)$.
By the relative satedness of $Y\overset{\pi }{\rightarrow }X $\ applied to
the further extension $Z\overset{\phi }{\rightarrow }Y$, $V (Z)$ and $ 
(\pi\circ\phi)^*(L^2(X ))$ are relatively orthogonal over $\phi^*(V(Y))$.
Lemma $\ref{ortho}$ gives the desired result.
\end{proof}

\begin{lemma}\label{lm:sated1}
\label{sated} Let $(A,\leq)$ be a directed set with no maximal element and $ 
((X_\alpha)_{\alpha\in A},(\pi_{\alpha_1,\alpha_2})_{\alpha_1,\alpha_2\in A,
\alpha_1\leq \alpha_2})$ be an inverse system of $\mathbf{PrbAlg}_\Gamma$-spaces with inverse limit $(X , (\pi_\alpha)_{\alpha\in A})$. Further
assume that for all $\alpha_1,\alpha_2\in A$ with $\alpha_1<\alpha_2$ the $ 
\mathbf{PrbAlg}_\Gamma$-morphism $\pi_{\alpha_1,\alpha_2}$ is relatively $V$-sated. Then the inverse limit $X $ is $V$-sated.
\end{lemma}

\begin{proof}
Each $\mathbf{PrbAlg}_\Gamma$-morphism $\pi_\alpha:X \rightarrow X _\alpha$
is relatively $V$-sated because we can factorize $\pi_\alpha=\pi_{\alpha,
\alpha^{\prime }}\circ \pi_{\alpha^{\prime }}$ for some $\alpha<\alpha^{ 
\prime }$ and apply Lemma \ref{relative sated}. Let $\psi:Y\rightarrow X $
be an arbitrary further $\mathbf{PrbAlg}_\Gamma$-morphism. For any $g\in
V(Y) $ and $f\in \bigcup_{\alpha\in A} \pi_\alpha^*(L^2(X _\alpha))$, we
have 
\begin{equation*}
\left\langle \psi^*f,g\right\rangle_{L^2(Y)}=\left\langle \psi^*P_{V(X
)}(f),g\right\rangle_{L^2(Y)}.
\end{equation*}
Since $\bigcup_{\alpha\in A} \pi_\alpha^*(L^2(X _\alpha))$ is dense in $ 
L^2(X )$, $f$ in the last equation can be replaced by any function in $L^2(X
)$. Thus, $V(Y)$ and $\psi^*(L^2(X ))$ are relatively orthogonal over $ 
\psi^*(V(X ))$, and so $X $ is $V$-sated.
\end{proof}

\begin{lemma}
\label{relative sated2} Every $\mathbf{PrbAlg}_\Gamma$-space $X $ has a
relatively $V$-sated extension.
\end{lemma}

\begin{proof}
We write all elements of $L^2(X )$ as $\{f_{\beta^{\prime
}}\}_{\beta^{\prime }<\alpha^{\prime }}$ for some limit ordinal number $ 
\alpha^{\prime }$. For each ordinal $\gamma<\alpha^{\prime }$, we let $ 
A_\gamma=\{f_{\beta^{\prime }}: \beta^{\prime }\leq\gamma\}$. Define a
well-ordering on the set $\bigcup_{\gamma<\alpha^{\prime }} \{\gamma\}\times
A_\gamma$ by the relation 
\begin{equation*}
(\gamma_1,f_{\beta_1^{\prime }})<(\gamma_2,f_{\beta_2^{\prime }})
\Longleftrightarrow \gamma_1<\gamma_2 \text{ or } (\gamma_1=\gamma_2 \text{
and } \beta_1^{\prime }<\beta_2^{\prime }).
\end{equation*}
Since $\bigcup_{\gamma<\alpha^{\prime }} \{\gamma\}\times A_\gamma$ is well
ordered, there exists an ordinal $\alpha$ and an order-preserving bijection $ 
\Psi$ from $\{\beta: \beta<\alpha\}$ to $\bigcup_{\gamma<\alpha^{\prime }}
\{\gamma\}\times A_\gamma$. Since $\alpha^{\prime }$ is a limit ordinal, for
each $\gamma<\alpha^{\prime }$ there is some $\gamma<\gamma^{\prime
}<\alpha^{\prime }$. As a result, there is no maximal element in $ 
\bigcup_{\gamma<\alpha^{\prime }} \{\gamma\}\times A_\gamma$; in other
words, $\alpha$ is a limit ordinal as well.

Let $\Phi=\Pi\circ\Psi$ where $\Pi$ is the projection mapping to the second
coordinate. For each $\gamma<\alpha$ and $f_{\beta^{\prime }}\in L^2(X )$,
we claim that there exists $\tau>\gamma$ such that $\Phi(\tau)=f_{\beta^{ 
\prime }}$. Suppose $\Psi(\gamma)=(\beta, g)$. We let $\beta_{\max}:=\max\{ 
\beta,\beta^{\prime }\}+1<\alpha^{\prime }$ and then $\Psi^{-1}(\beta_{ 
\max},f_{\beta^{\prime }})$ is the desired $\tau$. The interpretation is, when
enumerating $L^2(X )$ by $\Phi$, each function appears not only infinitely
many times but also arbitrarily late.

Resorting to transfinite induction, we construct a $\mathbf{PrbAlg}_\Gamma$-extension $X _\beta$ of $X $ for every $\beta<\alpha$ and a $\mathbf{PrbAlg} 
_\Gamma$-morphism $\phi_\gamma^\beta$ from $X _\beta$ to $X _\gamma$ for
every $\gamma<\beta<\alpha$. Set $X _\emptyset:=X $. Suppose $\epsilon\leq
\alpha$ is an ordinal and for each $\gamma<\beta<\epsilon$, $X _\beta$ and $ 
\phi_\gamma^\beta$ have been constructed.

Case 1: $\epsilon$ is the successor of $\epsilon-1$. For any $\mathbf{PrbAlg} 
_\Gamma$-extension $\eta:Z\to X _{\epsilon -1}$, we have 
\begin{equation*}
\|P_{V(Z)}((\phi_\emptyset^{\epsilon-1}\circ \eta)^*\Phi(\epsilon))\|_2-
\|P_{V(X _{\epsilon-1})}((\phi_\emptyset^{\epsilon-1})^*\Phi(\epsilon))\|_2
\leq \|\Phi(\epsilon)\|_2- \|P_{V(X
_{\epsilon-1})}((\phi_\emptyset^{\epsilon-1})^*\Phi(\epsilon))\|_2
\end{equation*}
since every orthogonal projection is a contraction. Hence, we find a $ 
\mathbf{PrbAlg}_\Gamma$-extension $\phi^{\epsilon}_{\epsilon-1}:X
_\epsilon\rightarrow X _{\epsilon-1}$ such that the difference 
\begin{equation*}
\|P_{V(X _{\epsilon})}((\phi_\emptyset^{\epsilon-1}\circ
\phi_{\epsilon-1}^{\epsilon})^*\Phi(\epsilon)) \|_2- \|P_{V(X
_{\epsilon-1})}((\phi_\emptyset^{\epsilon-1})^*\Phi(\epsilon)) \|_2
\end{equation*}
is at least half its supremum value over all extensions $\eta:Z\to X
_{\epsilon -1}$. For any $\gamma<\epsilon-1$, we set $\phi_\gamma^\epsilon:= 
\phi_\gamma^{\epsilon-1}\circ \phi^{\epsilon}_{\epsilon-1}$.

Case 2: $\epsilon$ is a limit ordinal. Let $(Z_\epsilon,(\psi_\beta^ 
\epsilon)_{\beta<\epsilon})$ be the inverse limit of the inverse system $((X
_\beta)_{\beta<\epsilon},(\phi_{\gamma}^ \beta)_{\gamma\leq \beta
<\epsilon})$. Let $\psi_{\epsilon}:X _\epsilon\rightarrow Z_{\epsilon}$ be a 
$\mathbf{PrbAlg}_\Gamma$-extension such that the difference 
\begin{equation*}
\|P_{V(X _{\epsilon})}((\psi_{\epsilon}\circ
\psi_\emptyset^{\epsilon})^*\Phi(\epsilon))\|_2 -
\|P_{V(Z_{\epsilon})}((\psi_\emptyset^{\epsilon})^*\Phi(\epsilon))\|_2
\end{equation*}
is at least half its supremum possible value over all extensions of $ 
Z_{\epsilon}$. Set $\phi_\gamma^\epsilon:=\psi_\gamma^{\epsilon}\circ
\psi_{\epsilon}$.

We now show that $\phi_\emptyset^\alpha:X _\alpha\rightarrow X $ is
relatively $V$-sated. Let $\pi:Y\rightarrow X _\alpha$ be an arbitrary
further extension. By Lemma $\ref{ortho}$, it is equivalent to showing that
for any $f\in L^2(X )$, 
\begin{equation*}
P_{V(Y)}((\phi_\emptyset^\alpha\circ \pi)^*f)= \pi^* P_{V(X
_\alpha)}((\phi_\emptyset^\alpha)^*f).
\end{equation*}
Since $\pi^*(V(X _\alpha))\subset V(Y)$, it suffices to show 
\begin{equation*}
\|P_{V(Y)}((\phi_\emptyset^\alpha\circ\pi)^*f)\|_2\leq \|P_{V(X
_\alpha)}((\phi_\emptyset^\alpha)^*f) \|_2.
\end{equation*}
Suppose for contradiction, $\|P_{V(Y)}((\phi_\emptyset^\alpha\circ
\pi)^*f)\|_2> \|P_{V(X _\alpha)}((\phi_\emptyset^\alpha)^*f) \|_2$. We know $ 
\|P_{V(X _\gamma)}((\phi_\emptyset^\gamma)^*f) \|_2$ is increasing in $ 
\gamma $ and bounded above by $\|f\|_2$. By the construction of $\Phi$, $f$
appears in the image of $\Phi$ infinitely many times. There exists an
ordinal $\gamma $ large enough such that $\Phi(\gamma)=f$ and one of the
following holds:

\begin{itemize}
\item[(i)] $\gamma$ is a successor and 
\begin{equation*}
\|P_{V(X _\gamma)}((\phi_\emptyset^\gamma)^*f) \|_2-\|P_{V(X
_{\gamma-1})}((\phi_\emptyset^{\gamma-1})^*f) \|_2<\frac{1}{2} 
\left(\|P_{V(Y)}((\phi_\emptyset^\alpha\circ \pi)^*f)\|_2- \|P_{V(X
_\alpha)}((\phi_\emptyset^\alpha)^*f)\|_2 \right).
\end{equation*}

\item[(ii)] $\gamma$ is a limit ordinal and 
\begin{equation*}
\|P_{V(X _\gamma)}((\phi_\emptyset^\gamma)^*f)
\|_2-\|P_{V(Z_{\gamma})}((\psi_\emptyset^{\gamma})^*f) \|_2<\frac{1}{2} 
\left(\|P_{V(Y)}((\phi_\emptyset^\alpha\circ \pi)^*f)\|_2- \|P_{V(X
_\alpha)}((\phi_\emptyset^\alpha)^*f) \|_2 \right).
\end{equation*}
\end{itemize}

Since $\|P_{V(X _\alpha)}((\phi_\emptyset^\alpha)^*f)\|_2\geq\|P_{V(X
_{\gamma-1})}((\phi_\emptyset^{\gamma-1})^*f) \|_2$ when $\gamma$ is a
successor and $\|P_{V(X _\alpha)}((\phi_\emptyset^\alpha)^*f)
\|_2\geq\|P_{V(Z_{\gamma})}((\psi_\emptyset^{\gamma})^*f)\|_2$ when $\gamma$
is a limit ordinal, we have 
\begin{equation*}
\|P_{V(X _\gamma)}((\phi_\emptyset^\gamma)^*f) \|_2-\|P_{V(X
_{\gamma-1})}((\phi_\emptyset^{\gamma-1})^*f) \|_2<\frac{1}{2} 
\left(\|P_{V(Y)}(( \phi_\emptyset^\alpha\circ \pi)^*f)\|_2- \|P_{V(X
_{\gamma-1})}((\phi_\emptyset^{\gamma-1})^*f) \|_2 \right),
\end{equation*}
or 
\begin{equation*}
\|P_{V(X _\gamma)}((\phi_\emptyset^\gamma)^*f)
\|_2-\|P_{V(Z_{\gamma})}((\psi_\emptyset^{\gamma})^*f) \|_2<\frac{1}{2} 
\left(\|P_{V(Y)}((\phi_\emptyset^\alpha\circ \pi)^*f )\|_2-
\|P_{V(Z_{\gamma})}((\psi_\emptyset^{\gamma})^*f)\|_2 \right),
\end{equation*}
either of which contradicts our choice of $X _\gamma$. Thus we have $ 
\phi_\emptyset^\alpha:X _\alpha\rightarrow X $ is relatively $V$-sated.
\end{proof}

Lemma \ref{lm:sated1} and Lemma \ref{relative sated2} give the following theorem.

\begin{theorem}
If $V$ is a functorial $L^2$-subspace of $\mathbf{PrbAlg}_\Gamma$-spaces,
then for every $\mathbf{PrbAlg}_\Gamma$-space $X=(X,\mu)$, there is a $ 
\mathbf{PrbAlg}_\Gamma$-morphism $\pi:Y\rightarrow X $ such that $Y=(Y,\nu)$
is $V$-sated.
\end{theorem}

\subsection{Extending factors relative to subgroups}

In this section, we show how to extend a factor map relative to a subgroup
to a factor map of the whole group. The corresponding result for countable
groups and standard Lebesgue spaces is \cite[Theorem 2.1] 
{austin2016nonconventional}.

\begin{theorem}
\label{ext} Let $\Gamma$ be an arbitrary discrete group, not necessarily
countable or amenable. Let $H$ be a subgroup of $\Gamma$. Let $X=(X,\mu,T)$
be a $\mathbf{PrbAlg}_\Gamma$-system and $Y=(Y,\nu,S)$ a $\mathbf{PrbAlg}_H$-system. Denote by $X_H=(X,\mu,T|_H)$ the $\mathbf{PrbAlg}_H$-system where $ 
T|_H$ is the restriction of the group homomorphism $T:\Gamma\to \mathrm{Aut} 
(X,\mu)$ to $H$. If $\beta:Y\to X_H$ is a $\mathbf{PrbAlg}_H$-morphism, then
there are a $\mathbf{PrbAlg}_\Gamma$-system $Z=(Z,\theta,R)$, a $\mathbf{ 
PrbAlg}_\Gamma$-extension $\pi:Z\to X$, and a $\mathbf{PrbAlg}_H$-extension $ 
\alpha:Z_H\to Y$ such that the diagram 
\begin{equation*}
\begin{tikzcd} & Y \arrow{dr}{\beta} \\ Z_H \arrow{ur}{\alpha}
\arrow{rr}{\pi} && X_H \end{tikzcd}
\end{equation*}
commutes in $\mathbf{PrbAlg}_H$.
\end{theorem}

\begin{proof}
We pass to the
canonical models $\tilde{X}=(\tilde{X},\mathcal{B}a(\tilde{X}),\tilde{\mu}, 
\tilde{T})$, $\tilde{Y}=(\tilde{Y},\mathcal{B}a(\tilde{Y}),\tilde{\nu}, 
\tilde{S})$, and $\tilde{\beta}$ of $X$, $Y$, and $\beta$, respectively. We construct $Z$, $\alpha$, and $\pi$ as follows. Firstly, we construct a ${\mathbf{CH}}_\Gamma$-system\footnote{We denote by $\mathbf{CH}$ the category of compact Hausdorff spaces and continuous maps, and ${\mathbf{CH}}_\Gamma$ denotes the category of topological dynamical $\Gamma$-systems formed on compact Hausdorff spaces and continuous factor maps, where the $\Gamma$-action is given by homeomorphisms.} $(\tilde{Z},\tilde{R})$
and ${\mathbf{CH}}_\Gamma$-maps $\tilde{\alpha}$ and $\tilde{\pi}$
satisfying a related commutative diagram in the dynamical category ${\mathbf{ 
CH}}_\Gamma$. Secondly, we construct a probability measure $\tilde\theta$ on $ 
(\tilde Z,\mathcal{B}a(\tilde{Z}))$ and show that it preserves the $\tilde{R}
$-action. Finally, we verify that this ${\mathbf{CHPrb}_{\Gamma}}$-system
satisfies the right commutative diagram. We can then map this diagram to the
dynamical categories of probability algebras using the deletion and
abstraction functors $\mathtt{Alg}$ and $\mathtt{Abs}$.

Step 1: we build a ${\mathbf{CH}}_\Gamma$-system $(\tilde{Z},\tilde{R})$.
Let 
\begin{equation*}
\tilde{Z}\coloneqq \{(y_\gamma)_\gamma\in \tilde{Y}^{\Gamma} :
y_{\gamma\eta} = \tilde{S}^{\eta^{-1}}y_\gamma \text{ and } \tilde{\beta} 
(y_\gamma) = \tilde{T}^{\gamma^{-1}}\tilde{\beta}(y_e) \text{ for all }
\gamma \in \Gamma, \eta \in H\}
\end{equation*}
where $e$ is the identity of $\Gamma$. Note that $\tilde{Z}$ is a compact
subspace of $\tilde{Y}^\Gamma$ (this basically follows from the fact that $ 
\tilde{S},\tilde{T}$ act by homeomorphisms and $\tilde{\beta}$ is
continuous). We can define a ${\mathbf{CH}}_\Gamma$-action $\tilde{R} 
:\Gamma\to \mathrm{Aut}(\tilde{Z})$ by 
\begin{equation*}
\tilde{R}^{\gamma^{\prime }}
((y_\gamma)_{\gamma\in\Gamma})=(y_{(\gamma^{\prime -1}\gamma})_{\gamma\in
\Gamma}.
\end{equation*}
(One easily checks that $\tilde{Z}$ is an $\tilde R$-invariant set so that $ 
\tilde R$ is well defined).

We set $\tilde{\alpha} : \tilde{Z}\rightarrow \tilde{Y}$ to be $\tilde{\alpha 
}((y_\gamma)_{\gamma\in \Gamma}) := y_e$ and $\tilde{\pi}:=\tilde{\beta}
\circ \tilde{\alpha}$. By construction, the diagram 
\begin{equation*}
\begin{tikzcd} & \tilde{Y} \arrow{dr}{\tilde\beta} \\ \tilde{Z}_H
\arrow{ur}{\tilde{\alpha}} \arrow{rr}{\tilde\pi} && \tilde{X}_H \end{tikzcd}
\end{equation*}
commutes in the dynamical category ${\mathbf{CH}}_H$. By construction of $ 
\tilde{Z}$, the map $\tilde{\pi}$ is also a ${\mathbf{CH}}_\Gamma$-factor
map.

Endow the space $\tilde{Z}$ with the Baire $\sigma$-algebra $\mathcal{B}a( 
\tilde{Z})$, which coincides with the restriction of $\mathcal{B}a(\tilde{Y} 
^\Gamma)=\mathcal{B}a(\tilde{Y})^{\otimes\Gamma}$ to $\tilde{Z}$ by \cite[ 
Lemma 2.1]{jt-etds}. In particular, the maps in the previous diagram preserve
Baire measurability.

Step 2: we construct a probability measure $\tilde{\theta}$ on $(\tilde{Z}, 
\mathcal{B}a(\tilde{Z}))$. Let $\{\nu_x\}_{x\in \tilde{X}}$ be the canonical
disintegration (see Theorem \ref{thm-disintegration}) of $\tilde{\nu}$ with respect to the factor map $\tilde{\beta} 
:\tilde{Y}\rightarrow \tilde{X}$. For each $\gamma\in \Gamma$ and $x\in 
\tilde{X}$, define $\tilde{\nu}_{\gamma,x}$ on $(\tilde{Y}^{\gamma H}, 
\mathcal{B}a(\tilde{Y}^{\gamma H}))$ by 
\begin{equation*}
\tilde{\nu}_{\gamma,x}(E):=\nu_x(\{y\in \tilde{Y}:(\tilde{S} 
^{\eta^{-1}}y)_{\gamma\eta\in \gamma H}\in E\}).
\end{equation*}
Since we can identify the Baire $\sigma$-algebra $\mathcal{B}a(\tilde{Y} 
^{\gamma H})$ with the product $\sigma$-algebra $\mathcal{B}a(\tilde{Y} 
)^{\otimes \gamma H}$, it follows that $\tilde{\nu}_{\gamma,x}$ is well
defined by first verifying cylinder sets and then applying the $\pi$-$ 
\lambda $ theorem.

By the axiom of choice, we pick a representative from each left coset $ 
\gamma H$ an element $\omega$ and denote their collection by $\Omega$. We
identify $\tilde{Y}^{\Gamma}=\Pi_{\omega\in \Omega}\tilde{Y}^{\omega H}$ so
as to define a probability measure 
\begin{equation*}
\tilde{\nu}_x^{\prime }:=\otimes_{\omega\in \Omega}\tilde{\nu}_{\omega, 
\tilde{T}^{\omega^{-1}}x}
\end{equation*}
on $\mathcal{B}a(\tilde{Y}^\Gamma)=\mathcal{B}a(\tilde{Y}^H)^{\otimes
\Omega} $. We show that the definition of $\tilde{\nu}_x^{\prime }$ is
independent from the choice of representatives. If $x\in X$, $A\in \mathcal{B 
}a(\tilde{Y}^{\gamma H})$, and $\gamma_1=\gamma_2\eta_1$ for some $\gamma_1,\gamma_2\in \Gamma$ and $\eta_1\in H$,
which means $ 
\gamma_1H=\gamma_2H$, then we have 
\begin{equation*}  
\begin{split}
\tilde{\nu}_{\gamma_1,\tilde{T}^{\gamma_1^{-1}}x}(E)&=\nu_{\tilde{T} 
^{\gamma_1^{-1}}x}(\{y:(\tilde{S}^{\eta^{-1}}y)_{\gamma_1\eta\in
\gamma_1H}\in E \}) \\
&=\nu_{\tilde{T}^{\eta_1^{-1}}\tilde{T}^{\gamma_2^{-1}}x}(\{y:(\tilde{S} 
^{\eta^{-1}}y)_{\gamma_2\eta_1\eta\in \gamma_2 H}\in E \}) \\
&=\tilde{S}^{\eta^{-1}_1}_{*}\nu_{\tilde{T}^{\gamma_2^{-1}}x}(\{y:(\tilde{S} 
^{\eta^{-1}}y)_{\gamma_2\eta_1\eta\in \gamma_2 H}\in E \}) \\
&=\nu_{\tilde{T}^{\gamma_2^{-1}}x}(\{y:(\tilde{S}^{\eta^{-1}}\tilde{S} 
^{\eta_1^{-1}}y)_{\gamma_2\eta_1\eta\in \gamma_2 H}\in E \}) \\
&=\tilde{\nu}_{\gamma_2,\tilde{T}^{\gamma_2^{-1}}x}(E).
\end{split} 
\end{equation*}
Hence, a finite product of the probability measures $\tilde{\nu} 
_{\gamma,\tilde T^{\gamma^{-1}}x}$ for $\gamma$ ranging from different left
cosets is independent from the choice of representatives. By the uniqueness
part of Carath\'eodory's extension theorem, we conclude that the product
probability measure $\tilde{\nu}_x^{\prime }$ is independent from the choice
of representatives.

Next, we define a measure $\tilde{\nu}_x$ on $(\tilde{Z},\mathcal{B}a(\tilde{ 
Z}))$ by 
\begin{equation*}
\tilde{\nu}_x(E\cap \tilde{Z})\coloneqq \tilde{\nu}_x^{\prime }(E)
\end{equation*}
for each $E\in \mathcal{B}a(\tilde{Y}^\Gamma)$. Note that $\tilde{Z}$ is a
closed subset of $\tilde{Y}^\Gamma$ but may not be Baire measurable.
Therefore, we need to check the well-definedness of $\tilde{\nu}_x$. 
It suffices to show that 
\begin{equation*}
E\in \mathcal{B}a(\tilde{Y}^\Gamma) \text{ and }E\cap \tilde{Z}=\emptyset
\Rightarrow \tilde{\nu}_x^{\prime }(E)=0.
\end{equation*}
Since $\mathcal{B}a(\tilde{Y}^\Gamma)=\mathcal{B}a(\tilde{Y} 
)^{\otimes\Gamma} $, $E$ depends on only countably many coordinates. Hence
there exists $\{\gamma_i\}_{i=1}^{\infty}$ such that $E=E^{\prime }\times
\otimes_{\gamma\in \Gamma\backslash\{\gamma_i\}_{i=1}^{\infty}}\tilde{Y}$
where $E^{\prime }\in \mathcal{B}a(\tilde{Y})^{\otimes \mathbb{N}}$. Let 
\begin{equation*}
\tilde{Z}^*=\{(y_\gamma)_\gamma\in \tilde{Y}^\Gamma: y_{\gamma_i\eta}=\tilde{ 
S}^{\eta^{-1}}y_{\gamma_i} \text{ and }\tilde{\beta}(y_{\gamma_i})=\tilde{T} 
^{\gamma_i^{-1}}\tilde{\beta}(y_e)\; \text{ for all } i\geq 1 \text{ and }
\gamma_i\eta\in \{\gamma_i: i\geq 1\}\}.
\end{equation*}
Since $\tilde{Z}\cap E=\emptyset$ implies $\tilde{Z}^*\cap E=\emptyset$, in
order to show $\tilde{\nu}^{\prime }_x(E)$=0, it suffices to show $\tilde{\nu 
}^{\prime }_x(\tilde{Z}^*)=1$ (since $\tilde{Z}^*$ only depends on countable
many coordinates, it is guaranteed to be Baire measurable). We group the $ 
\gamma_i$ according to the left cosets $\omega H$ they belong to. So suppose 
$\{\gamma_i\}=\bigcup_i \{\omega_i\eta_{i,j}\}_j$, where each $\eta_{i,j}\in
H$, $\omega_i\in \Omega$. For each $\omega_i$, we have 
\begin{equation*}
\tilde{\nu}_{\omega_i,\tilde{T}^{\omega_i^{-1}}x}(\{(y_{\omega_i\eta})_{ 
\omega_i\eta\in \omega_iH}: y_{\omega_i\eta_{i,j}}=\tilde{S} 
^{\eta_{i,j}^{-1}}y_{\omega_i}, \tilde{\beta}(y_{\omega_i})= \tilde{T} 
^{\omega_i^{-1}}x\})=1.
\end{equation*}
Note that 
\begin{equation*}
\otimes_{i=1}^\infty\{(y_{\omega_i\eta})_{\omega_i\eta\in \omega_iH}:
y_{\omega_i\eta_{i,j}}=\tilde{S}^{\eta_{i,j}^{-1}}y_{\omega_i}, \tilde{\beta} 
(y_{\omega_i})= \tilde{T}^{\omega_i^{-1}}x\}\times \otimes_{\omega\neq
\omega_i}\tilde{Y}^{\omega H}\subset \tilde{Z}^*.
\end{equation*}
Thus $\tilde{\nu}_x^{\prime }(\tilde{Z}^*)=1$ and consequently, $\tilde{\nu} 
_x$ is well defined.

For any set $A=A^{\prime }\cap \tilde{Z}\in \mathcal{B}a(\tilde{Z})$ where $ 
A^{\prime }\in \mathcal{B}a(\tilde{Y}^\Gamma)$, we aim to prove the mapping $ 
x\mapsto \tilde{\nu}_x(A)$ is Baire measurable. Suppose there are $ 
\omega_1,\ldots,\omega_m\in \Omega$, $\eta_{i,1}, \ldots,\eta_{i,n_i}\in H$
for each $i\leq m$, and $A_{i,j}\in \mathcal{B}a(\tilde{Y})$ for all $i\leq
m $ and $j\leq n_i$ such that $A^{\prime
}=\{(y_\gamma)_\gamma:y_{\omega_i\eta_{i,j}}\in A_{i,j}\; \forall i\leq m,
j\leq n_i\}$. Then 
\begin{equation}  \label{cross_session}
\begin{split}
\tilde{\nu}_x^{\prime }(A^{\prime })&=\prod_{i=1}^m \tilde{\nu}_{\omega_i, 
\tilde{T}^{\omega_i^{-1}}x}(\{(y_{\omega_i\eta})_{\eta\in
H}:y_{\omega_i\eta_{i,j}}\in A_{i,j} \text{ for any } j\leq n_i \}) \\
&=\prod_{i=1}^m \nu_{\tilde{T}^{\omega_i^{-1}}x}(\tilde{S} 
^{\eta_{i,1}}(A_{i,1})\cap\cdots\cap \tilde{S}^{\eta_{i,n_i}}(A_{i,n_i})).
\end{split} 
\end{equation}
Since $x\mapsto \nu_x$ is Baire measurable and the product of finitely many
Baire measurable functions is still Baire measurable, it follows that $ 
x\mapsto \tilde{\nu}_x(A)$ is Baire measurable whenever $A^{\prime }\in 
\mathcal{B}a(\tilde{Y}^\Gamma)$, as the cylinder sets generate $\mathcal{B}a( 
\tilde{Y}^\Gamma)$. As a result, we are able to define 
\begin{equation*}
\tilde{\theta}:=\int_{\tilde{X}} \tilde{\nu}_xd\tilde{\mu}(x).
\end{equation*}
Observe that each $\tilde{\nu}_x(\tilde{Z})=\tilde{\nu}_x^{\prime }(\tilde{Y} 
^{\Gamma})=1$. Therefore, $\tilde{\theta}$ is a probability measure as well.

Step 3: we verify that $(\tilde{Z},\mathcal{B}a(\tilde{Z}),\tilde{\theta}, 
\tilde{R})$ is a ${\mathbf{CHPrb}_{\Gamma}}$-system satisfying the desired
diagram.

We claim that $\tilde{R}$ is a measure-preserving transformation. Suppose $ 
\gamma^{\prime }\in \Gamma$, $x\in \tilde X$, $\omega_1,\ldots,\omega_m\in
\Omega$, $\eta_{i,1},\ldots,\eta_{i,n_i}\in H$ for each $i\leq m$, and $ 
A_{i,j}\in \mathcal{B}a(\tilde{Y})$ for all $i\leq m$ and $j\leq n_i$. Then
we obtain 
\begin{eqnarray*}
&& \tilde{R}_{*}^{\gamma^{\prime }} \tilde{\nu}_x^{\prime
}(\{(y_\gamma)_\gamma: y_{\omega_i\eta_{i,j}}\in A_{i,j}\;\forall i\leq m,
j\leq n_i\}) \\
&=& \tilde{\nu}_x^{\prime }(\{(y_\gamma)_\gamma: y_{\gamma^{\prime
-1}\omega_i\eta_{i,j}}\in A_{i,j}\;\forall i\leq m, j\leq n_i\}) \\
&=& \prod_{i=1}^m \nu_{\tilde{T}^{\omega_i^{-1}\gamma^{\prime }}x}^{\prime }( 
\tilde{S}^{\eta_{i,1}}(A_{i,1})\cap\cdots \cap \tilde{S}^{ 
\eta_{i,n_i}}(A_{i,n_i})) \\
&=& \tilde{\nu}_{\tilde{T}^{\gamma^{\prime }} x}^{\prime
}\{(y_\gamma)_\gamma: y_{\omega_i\eta_{i,j}}\in A_{i,j}\;\forall i\leq m,
j\leq n_i\},
\end{eqnarray*}
where the last two equalities follow from two applications of  
\eqref{cross_session} (while in the first applications we work with family
of representatives $\gamma^{\prime -1}\Omega$ instead of $\Omega$).
Therefore, $\tilde{R}^{\gamma^{\prime }}_{*}\tilde{\nu}_x^{\prime }=\tilde{ 
\nu}_{\tilde{T}^{\gamma^{\prime }}x}^{\prime }$. By the definition of $ 
\tilde{\nu}_x$, $\tilde{R}^{\gamma^{\prime }}_{*}\tilde{\nu}_x=\tilde{\nu}_{ 
\tilde{T}^{\gamma^{\prime }}x}$. Since $\tilde{\mu}$ is $\tilde{T}$ 
-invariant, integrating $\tilde{R}^{\gamma^{\prime }}_{*}\tilde{\nu}_x= 
\tilde{\nu}_{T^{\gamma^{\prime }}x}$ over $\tilde\mu$ gives $\tilde{R} 
^{\gamma^{\prime }}_{*}\tilde{\theta}=\tilde{\theta}$.

Recall that the map $\tilde{\alpha}$ is a ${\mathbf{CH}}_H$-factor map.
Moreover, for any $x\in \tilde{X}$, we have $\tilde{\alpha}^*\tilde{\nu} 
_x=\nu_x$ by observing that for any $A\in \mathcal{B}a(\tilde{Y})$, 
\begin{equation*}
\tilde{\nu}_x(\alpha^{-1}A)=\tilde{\nu}_x^{\prime }(A\times
\otimes_{\gamma\neq e,\gamma\in \Gamma}Y)=\nu_x(A).
\end{equation*}
Therefore, 
\begin{equation}  \label{eq-pushforward}
\tilde{\alpha}_{*}\tilde{\theta}=\int_{\tilde{X}} \tilde{\alpha}_{*}\tilde{ 
\nu}_x\tilde{\mu}(dx)=\int_{\tilde{X}} \nu_x\tilde{\mu}(dx)=\tilde{\nu},
\end{equation}
which shows that $\tilde{\alpha}$ is a ${\mathbf{CHPrb}}_H$-factor map.

It remains to show that $\tilde\pi$ is a ${\mathbf{CHPrb}_{\Gamma}}$-factor
map, but this is a direct consequence of \eqref{eq-pushforward}: 
\begin{equation*}
\tilde{\pi}_{*}\tilde{\theta}=\beta^*\tilde{\nu}=\tilde{\mu}.
\end{equation*}
\end{proof}

\end{document}